\documentclass[10pt]{article}  

\usepackage{version}
\excludeversion{JAPANESE}

\usepackage{amsthm}
\usepackage{etoolbox}

\theoremstyle{plain}
\theoremstyle{definition}
\ifdefined\theoremstyleJapanese  %
  \newtheorem{axiom}{公理}[section] 
  
  \newtheorem{theorem}{定理}[section] 
  
  \newtheorem{lemma}{補題}[section]
  
  \newtheorem{proposition}{命題}[section]

  \newtheorem{corollary}{系}[section]
  
  \newtheorem{claim}{Claim}[section]
  
  \newtheorem{conjecture}{Conjecture}[section]
  
  \newtheorem{hypothesis}{仮説}[section]

  \newtheorem{assumption}{仮定}[section]

  \newtheorem{remark}{注意}[section]
  
  \newtheorem{example}{例}[section]
  
  \newtheorem{examples}{例}[section]
  
  \newtheorem{program}{プログラム}[section]
  
  \newtheorem{prompt}{プロンプト}[section]
  
  \newtheorem{breakpoint}{一時停止}[section]
  
  \newtheorem{problem}{問題}[section]
  
  \newtheorem{solution}{解}[section]
  
  \newtheorem{fact}{事実}[section]
  
  \newtheorem{principle}{原理}[section]
  
  \newtheorem{definition}{定義}[section]
  
  \newtheorem{definitions}{定義系}[section]
  
  \newtheorem{questioN}{問題}[section]
  
  \newtheorem{answer}{解答}[section]
  
  \newtheorem{exercise}{練習問題}[section]

  \newtheorem{notation}{表記}[section]
  
  \newtheorem{algorithm}{アルゴリズム}[section]
  
\else  %
 
\newtheorem{theorem}{Theorem}[section] 
\newtheorem{lemma}{Lemma}[section]
\newtheorem{proposition}{Proposition}[section]
\newtheorem{corollary}{Corollary}[section]

\newtheorem{assumption}{Assumption}[section]
\newtheorem{remark}{Remark}[section]

\newtheorem{definition}{Definition}[section]

\newtheorem{notation}{Notation}[section]
\newtheorem{algorithm}{Algorithm}[section]
\fi

\AtEndEnvironment{axiom}{$\blacktriangleleft$}
\AtEndEnvironment{theorem}{$\blacktriangleleft$}
\AtEndEnvironment{lemma}{$\blacktriangleleft$}
\AtEndEnvironment{proposition}{$\blacktriangleleft$}
\AtEndEnvironment{corollary}{$\blacktriangleleft$}
\AtEndEnvironment{claim}{$\blacktriangleleft$}
\AtEndEnvironment{conjecture}{$\blacktriangleleft$}
\AtEndEnvironment{hypothesis}{$\blacktriangleleft$}
\AtEndEnvironment{assumption}{$\blacktriangleleft$}
\AtEndEnvironment{remark}{$\blacktriangleleft$}
\AtEndEnvironment{example}{$\blacktriangleleft$}
\AtEndEnvironment{examples}{$\blacktriangleleft$}
\AtEndEnvironment{problem}{$\blacktriangleleft$}
\AtEndEnvironment{solution}{$\blacktriangleleft$}
\AtEndEnvironment{fact}{$\blacktriangleleft$}
\AtEndEnvironment{principle}{$\blacktriangleleft$}
\AtEndEnvironment{definition}{$\blacktriangleleft$}
\AtEndEnvironment{definitions}{$\blacktriangleleft$}
\AtEndEnvironment{questioN}{$\blacktriangleleft$}
\AtEndEnvironment{answer}{$\blacktriangleleft$}
\AtEndEnvironment{exercise}{$\blacktriangleleft$}
\AtEndEnvironment{notation}{$\blacktriangleleft$}
\AtEndEnvironment{algorithm}{$\blacktriangleleft$}

\newcommand{\underbarA}[1]{ \underline{#1} }  %
\newcommand{\sigmaalgebra}[0]{ \mathcal{F} }  
\newcommand{\sigmaalgebrat}[1]{ \mathcal{F}_{#1} }  
\newcommand{\filtration}[3]{ ({#1}_{#2})_{{#2}\in{#3}} } 

\newcommand{\stoppingTimeSet}[1]{ \mathcal{T}_{#1} } 
\newcommand{\TimeSet}[0]{ \mathbb{T} }  

\newcommand{\abs}[1]{ \left|#1\right| }  

\newcommand{\ProbabilityMeasureP}[0]{ \mathbb{P}     }

\newcommand{\expectationx}[2]{ \mathbb{E}_{#1} \left[#2\right]    }

\newcommand{\expectationCond}[2]{ \mathbb{E}_{} \left[#1\middle|#2\right]    }  
  
\newcommand{\ExpectationCond}[2]{ \mathbb{E}_{} \Biggl[#1\Biggm|#2\Biggr]    }

\newcommand{\process}[2]{ {#1}_{#2}  }  
 
\newcommand{\discount}[1]{ e^{#1}  }  
  
\newcommand{\valuefuncA}[3]{ {#1}(#2, #3)  } 
\newcommand{\valuefuncB}[4]{ {#1}(#2, #3, #4)  }  
\newcommand{\valuefuncC}[5]{ {#1}(#2, #3, #4, #5)  } 
\newcommand{\intA}[4]{ \int_{#1}^{#2} {#3} d {#4}  } 
\newcommand{\ousde}[6]{ d{#1}_{#2}=-{#3}({#1}_{#2}-{#4}) d {#2} + {#5} d{#6}_{#2}}  
\newcommand{\ousdeA}[4]{ d{#1}=-{#2}{#1}d{#3}+\sqrt{#2} d  {#4}  }  

\newcommand{\generator}[1]{ \mathcal{L}{#1}  }  %
\newcommand{\generatorx}[2]{ \mathcal{L}_{#1}{#2}  }  %
\newcommand{\DEOUhomo}[2]{ {#1}'' -2 {#2}{#1}' -\frac{2\delta}{\theta}{#1}  }  %
\newcommand{\DEOUnonhomo}[3]{ \DEOUhomo{#1}{#2} -{#3}\frac{2\sigma}{\sqrt\theta}{#2}  }  %
\newcommand{\Hermite}[2]{ H_{#1}(#2)  }  
\newcommand{\HermiteNu}[1]{ \Hermite{\nu}{#1}  }  
\newcommand{\HermiteNui}[1]{ \Hermite{\nu-1}{#1}  }

\newcommand{\HermiteNuIntegralRepresentation}[1]{ \frac{1}{2\Gamma(-\nu)}\intA{0}{\infty}{e^{-s-2{#1}^{\sqrt{s}}}s^{-\frac{1}{2}\nu-1}}{s}  }  %
\newcommand{\HermiteNuIntegralRepresentationA}[1]{ \frac{1}{\Gamma(-\nu)}\intA{0}{\infty}{e^{-t^2-2t{#1}}{t}^{-\nu-1}}{t}  }  %
\newcommand{\Wronskian}[1]{ \mathcal{W}\{{#1}\}  }  %
\newcommand{\interior}[1]{ \mathbf{Int}\left({#1}\right)  }  %

\newcommand{\indicator}[1]{ \mathbf{1}_{\{#1\}}  }  %
\newcommand{\real}[1]{ \mathbb{R}^{#1}  }  %
\newcommand{\continuationRegion}[1]{ {\mathcal{C}}_{#1}^{\xi}}   %
\newcommand{\continuationRegionA}[2]{ {\mathcal{C}}_{#1}^{#2}}   %
\newcommand{\continuationAfterSwitchingRegionA}[2]{ {\mathcal{O}}_{#1}^{#2}}   %
\newcommand{\stoppingRegion}[1]{ {\mathcal{S}}_{#1}^{\xi}}  %
\newcommand{\stoppingRegionA}[2]{ {\mathcal{S}}_{#1}^{#2}}  %
\newcommand{\stoppingRegionCandidate}[2]{{\mathcal{Q}}^{#1}_{#2}}  %
\newcommand{\continuousFunction}[1]{ {{C}}^{#1}}  
  
\newcommand{\positionspace}[0]{ \mathbb{I}  } 
\newcommand{\controlspace}[0]{ \mathcal{A}{}  }

\newcommand{\controlspaceAn}[1]{ \mathbb{A}_{#1}  }

\newcommand{\regionO}[0]{ {\mathcal{O}}}   

\newcommand{\superjets}[2]{ J_{#2}^{2,+}{#1} } 
\newcommand{\subjets}[2]{ J_{#2}^{2,-}{#1} }

\newcommand{\subsuperjets}[2]{ J_{#2}^{2,\mp}{#1} }

\newcommand{\wrt}[0]{{with respect to} }

\newcommand{\SDE}[0]{ stochastic differential equation }

\newcommand{\TestFuncMin}[3]{\{{#1}, {#2}\} \smile {#3}}
\newcommand{\TestFuncMax}[3]{\{{#1}, {#2}\} \frown {#3}}

\newcommand{\setViscositySolution}[1]{\bigg[ {#1} \bigg]}
\newcommand{\setViscositySuperSolution}[1]{\bigg\lceil {#1} \bigg\rceil}
\newcommand{\setViscositySubSolution}[1]{\bigg\lfloor {#1} \bigg\rfloor}

\makeatletter
\newcommand{\subsubsubsection}{\@startsection{paragraph}{4}{\z@}%
  {1.0\Cvs \@plus.5\Cdp \@minus.2\Cdp}%
  {.1\Cvs \@plus.3\Cdp}%
  {\reset@font\sffamily\normalsize}
}
\makeatother
\newcommand{\citeN}[1]{\cite{#1}}   

\usepackage{amsxtra} 
\usepackage{amsfonts} 
\usepackage{amssymb}  
\usepackage{latexsym} 

\begin{JAPANESE}
  \usepackage{plext} 
\end{JAPANESE}
\usepackage{psfrag}
\usepackage{verbatim}
\usepackage{fancybox}
\usepackage{makeidx}
\usepackage{oldgerm}
\usepackage{lscape}
\usepackage{multirow}
\usepackage{array}
\usepackage{subfigure}

\usepackage{graphicx}          
\usepackage{hyperref}   

\usepackage{setspace}

\usepackage[matrix, curve, arrow, tips, frame]{xy}

\usepackage{listings}
\usepackage{tikz}
\usetikzlibrary{positioning}

\excludeversion{comment2} 
\excludeversion{copiedTheorems}

\excludeversion{JapaneseComm}

\includeversion{nonconference}
\excludeversion{conference}

\excludeversion{omitlevel1}
\excludeversion{omitlevel2}
\excludeversion{omitlevel3}

\title{
A viscosity solution as a piecewise classical solution to a free boundary problem for the optimal switching problem with simultaneous multiple switches
  }

\author{
  Kiyoshi Suzuki \\
  Center for Data Science and AI Innovation \\
  Graduate School of Economics, Shiga University\\
  \texttt{kiyoshi-suzuki@biwako.shiga-u.ac.jp}\\
  ORCID: \href{https://orcid.org/0000-0003-1133-0649}{0000-0003-1133-0649}
}
\date{\today}

\begin{omitlevel1}
\end{omitlevel1}

\begin{document}    

\maketitle  


\begin{abstract} %

\citeN{suzuki2020optimal} proves the uniqueness of the viscosity solution to a variational inequality which is solved by the value function of the infinite horizon optimal switching problem  with simultaneous multiple switchings.
Although it also identifies each connected region possibly including at most one connected switching region, the exact switching regions of the solution are not identified.  The problem is finally converted into a system of free boundary problems and generally solved by the numerical calculation.
However, if the PDE part of the variational inequality has a classical solution, the viscosity solution may be constructed as a series of piecewise classical solutions, possibly analytical.

Under a certain assumption we prove that the series of piecewise classical solutions is indeed the viscosity solution on $\real{}$, after we prove the smooth pasting condition is its necessary condition, and establish the algorithm to compute all the free boundaries.  Applying the results to the concrete problem studied in \citeN{suzuki2020optimal} we find the explicit solution and identify the continuation and switching regions in a computer with Python programs.
\end{abstract}

\section{Introduction}
The concept of viscosity solutions was first introduced by Crandall and Lions in the early 1980s as a generalized solution framework for nonlinear first- and second-order partial differential equations (PDEs), particularly Hamilton–Jacobi–Bellman (HJB) equations (See \citeN{CrandallIshiiLions:92}). Unlike classical solutions, which require differentiability, viscosity solutions allow for meaningful interpretation and analysis even when the solution lacks smoothness, making the approach highly robust for dealing with value functions in optimal control and differential game problems.

Since its inception, the viscosity solution theory has found extensive applications in mathematical finance, especially in the context of dynamic portfolio optimization, real options, and optimal stopping and switching problems. The framework is particularly well-suited for stochastic control problems involving discontinuities, regime changes, or singularities—situations where classical PDE methods often fail.

In financial applications, the HJB equations governing optimal investment and consumption strategies frequently lack smooth solutions due to market frictions, transaction costs, or stochastic volatility. Viscosity solutions provide a rigorous analytical and numerical foundation for such problems. The methodology has been widely adopted in both theoretical and applied finance literature, underpinning models of optimal asset allocation, market entry/exit decisions, option pricing under transaction costs, and investment under uncertainty.

Classical solutions describe necessary conditions for the value function only within the continuation region, where the solution is smooth. However, it is essential to verify—via a verification theorem—that a candidate function coincides with the true value function.

Moreover, classical solutions often encounter domain discontinuities at the boundary of the continuation region. In contrast, viscosity solutions allow for a simply connected domain covering the entire region of interest and can represent the solution as a globally continuous function across the domain. In the existing literature, classical verification theorems are used in studies such as \citeN{ZhangZhang:08}, \citeN{GuoZhang:05}, and \citeN{Zervos:03}.

To avoid the complexity inherent in the classical approach, the concept of viscosity solutions was introduced by Crandall and Lions. This approach has provided a powerful and general-purpose analytical framework for stochastic control problems, allowing for a mathematically rigorous interpretation of solutions to formal Bellman equations under only local boundedness conditions. Foundational works on viscosity solutions include \citeN{CrandallIshiiLions:92} and \citeN{Fleming:93}.

Applications of the viscosity solution framework to optimal liquidation and optimal switching problems can be found in \citeN{PemyZhang:06}, \citeN{PhamVathZhou:09}, \citeN{PhamVath:07}, and \citeN{Pham:09}.

In the present study, we also adopt the viscosity solution approach. When uniqueness of the solution can be established, there is no need to invoke a classical verification theorem. It is also of critical importance for numerical computation that the value function to be obtained is represented as the unique viscosity solution defined over the entire real space.

\citeN{suzuki2020optimal} proves the uniqueness of the viscosity solution to a Hamilton--Jacobi--Bellman (HJB) variational inequality which is solved by the value function of the infinite horizon optimal switching problem with simultaneous multiple switchings.
The value function on $\real{}$ for the particular position and the number of remaining switchings is generally composed of continuation regions and switching regions. Although the value function as a viscosity solution is of class $\continuousFunction{1}$ on all of $\real{}$, the problem is that the formula of the function is not the same on $\real{}$, i.e., the function on $\real{}$ is composed of piecewise analytical functions which have different formulas.
You should find the analytical solution on each continuation region for a particular PDE, but beyond the free boundaries of the continuation region, you should solve other DPEs for other viscosity solutions and connect them.
The piecewise structure of the viscosity solution to the variational inequality is the major difficulty of the solution process of the problem.  
This type of solution formula is different from what is called closed analytical solution, nor the numerical solution.

We provide a theorem stating that if a classical solution to the PDE part of the HJB-variational inequality coincides with the unique viscosity solution to the variational inequality on the closure of the continuation region, then the viscosity solution is continuously differentiable at the free boundaries, and also its converse theorem describing a sufficient condition of a series of piecewise classical solutions to be the viscosity solution.
Using these theorems, we provide an algorithm to efficiently  find for each value function all the connected piecewise regions each of which is governed by a particular form of analytical function and calculate them using Python on a computer.

The remainder of this paper is organized as follows. 
Section~\ref{sec:assumptions_problem} introduces the assumptions and definitions underlying the optimal switching problem, including the structure of the value function, variational inequality formulation, and the decomposition of the state space into continuation and switching regions.
Section~\ref{sec:Piecewise classical viscosity solution} presents the theoretical foundation for constructing the viscosity solution as a series of piecewise classical solutions. We establish the necessary and sufficient smooth pasting conditions and provide a general theorem ensuring that the constructed piecewise functions yield the unique viscosity solution over the entire domain.
Section~\ref{sec:example_model} applies the theoretical results to a concrete example of a mean-reverting pair-trading model, as studied in \citeN{suzuki2020optimal}. We formulate the specific switching problem, derive the analytical general solutions on each continuation region, and apply the developed algorithm to compute the connected regions and identify the free boundaries.
Section~\ref{sec:VF_algorithm_datastr} describes the implementation of the algorithm in Python and presents the resulting value function graphs up to $n = 5$. All solutions are analytically derived and computed using Python without relying on numerical solvers.
Finally, Section~\ref{sec:conclusion} concludes the paper.

\section{Assumptions for the Value Function}\label{sec:assumptions_problem}

\subsection{Properties of the Value Function} \label{sec:PropertyValueFunc}
\def\continuationRegionN{\continuationRegion{n}}
\def\stoppingRegionN{\stoppingRegion{n}}

\def\CC#1#2{C^{#1 #2}_n}
\def\vf{\valuefuncB{v}{z}{\xi}{n}}
\def\vfhat{\valuefuncA{\widehat{v}}{z}{\xi}}
\def\vfhatxi#1{\valuefuncA{\widehat{v}}{z}{#1}}
\def\vfhatxiz#1#2{\valuefuncA{\widehat{v}}{#2}{#1}}



Let $(\Omega, \mathcal{F}, \mathbb{F}, \ProbabilityMeasureP) \label{def:probability_Space}$ be a complete probability space with the filtration  $\mathbb{F}=\filtration{\sigmaalgebra}{s}{\TimeSet}$, generated by a one-dimensional diffusion process $\{Z_t\in\real{}: t\in\TimeSet\}$ with its infinitesimal generator $\generatorx{}{}$, where $\TimeSet\equiv[t_0, \infty)$ and  $\mathcal{F}_{\infty}=\mathcal{F}$. 
The set of switching states (regimes) $\positionspace$ is composed of different positions (regimes)  for which the investor can choose.
$\controlspace{(\xi)}\subset\positionspace$ denotes the feasible set of the states (positions) into which the current state $\xi\in\positionspace\; (\xi\notin\controlspace{(\xi)})$ can transfer. 

$\stoppingTimeSet{t_0}$  denotes the feasible set of stopping times $\{\tau_i\}$, $(\tau_0=t_0\leq \tau_i \leq \tau_{i+1} (i=1, 2, \cdots); \tau_i\to\infty \; (i\to+\infty) \quad a.s.)$ for the switching time of the strategy, 
valued in $\TimeSet$.
  The stopping time is based on the filtration $\mathbb{F}=\filtration{\sigmaalgebra}{s}{\TimeSet}$, i.e.,   $\{ \omega \in \Omega | t_0\leq {\tau_i}(\omega)\leq s \} \in \sigmaalgebrat{s}, \; s\in \TimeSet$. At each time $\tau_i$, the position is switched from $\xi_{\tau_{i-1}}$ to  $\xi_{\tau_{i}}\ne\xi_{\tau_{i-1}}$; i.e., 
$\xi_{\tau_i} \in \controlspace{(\xi_{\tau_{i-1}})} \; (i=1,2,\cdots)$ is a $\sigmaalgebrat{\tau_i}-$measurable random variable.
The switching process $\xi_{s} (s\in \TimeSet)$ is a piecewise-constant process, formulated as:
$\xi_s=\displaystyle\sum_{i=0}^{\infty} \xi_{\tau_i}\indicator{[\tau_i, \tau_{i+1})}(s), \; s\in \TimeSet.  \label{eq:switchingProcess}$
A particular strategy is specified in the form of $\{\xi_{\tau_i}, \tau_i \}, ( \tau_i \in \stoppingTimeSet{t_0})$. 
The feasible set of switching strategies (control space) is defined as follows for the number of remaining switches $n\in \{0,1,\cdots,\infty\}$ with the initial states $\xi\in\positionspace$:
$\controlspaceAn{n}(\xi)\equiv \{\{{\xi}_{\tau_i}, \tau_i\} | {\xi}_{\tau_i} \in \controlspace({\xi}_{{\tau}_{i-1}}) \subset \positionspace, 
    \tau_i\in\stoppingTimeSet{t_0}, \xi_{\tau_0}=\xi, i=1, \cdots, n  \}. \label{def:controlspaceA}$
  
The value function $\valuefuncB{v}{z}{\xi}{n}$ with the current state $Z_{t}^-=z\in\real{}, t\in\TimeSet$ under regime $\xi_{t}^-=\xi\in\positionspace$ with the number of options to switch $n\geq 0$ is defined as the supremum of the criterion function $J$ \wrt the strategy $\alpha$ over the feasible set of strategies   $\controlspaceAn{n}(\xi)$ as follows.
\begin{align} 
  \valuefuncB{v}{z}{\xi}{n} \equiv
  \begin{cases}
    \sup_{\alpha \in \controlspaceAn{n}(\xi)} \valuefuncC{J}{z}{\xi}{n}{\alpha},   & n\geq 1,\\
    \valuefuncC{J}{z}{\xi}{0}{\phi},   & n=0.
  \end{cases} \label{def:valueFunc_criterion}
\end{align}

\begin{assumption}[Linear growth value function:] \label{theo:finite}
  There exists a constant $D>0$, such that for all $z\in\real{}$ and integer $n\geq 0$, the following inequality holds:
\begin{align}
  &\vf \leq D(1+|z|). \label{cond:vflinear} 
\end{align}
\end{assumption}

\begin{assumption}[Boundedness from below of the value function: ]
  \label{theo:boundedVf}
  For all $z\in \real{}, \xi\in\positionspace$ and integer $n\geq 1$, 
  \begin{align}
      \exists D\in\real{}, \valuefuncB{v}{z}{\xi}{n}\geq D.
  \end{align}
\end{assumption}

\begin{omitlevel1}
\begin{assumption}[Bounded and convergence of value function: \citeN{suzuki2016optimal}, Corollary 3.1] \label{coro:convergenceVF}~\\
  任意の有限な$z\in\real{}$ ($\abs{z} < \infty$)に対して以下が成立する。
  \begin{enumerate}
    \item [(i).] 任意の整数$n\leq \infty$に対して、
    \begin{align}
      \abs{\vf} \leq  C_z \label{ineq:boundedConvergence} 
    \end{align}
    を満たす定数$C_z<\infty$を$z$に応じて決めることができる。
    \item [(ii).]   $\lim_{n\to\infty}  \valuefuncB{v}{z}{\xi}{n}$
    が有限値に収束する。
  \end{enumerate}
\end{assumption}
\end{omitlevel1}

\def\PDE#1#2{PDE\{#1, #2\}}
\def\SDE#1#2{SDE\{#1, #2\}}

\begin{omitlevel1}
\begin{definition}[Test function: \citeN{suzuki2020optimal}] \label{def:testFunction}~\\
For a region ${\mathcal{O}}\subset\real{}$, suppose there exists an open neighborhood $\tilde{\mathcal{O}}\supset\mathcal{O}$ of $\mathcal{O}$.
For any function $v: x\mapsto\real{}$ defined on $\tilde{\mathcal{O}}$, if there exists a smooth function $\varphi\in\continuousFunction{2}(\tilde{\mathcal{O}})$ minimizes (maximizes, respectively) $v(x)-\varphi(x)$ relative to $\tilde{\mathcal{O}}$ at  $\bar{x}\in\mathcal{O}$ and $v(\bar{x})=\varphi(\bar{x})$, i.e., 
\begin{align}
  \min_{x\in\tilde{\mathcal{O}}\supset\mathcal{O}}\{v(x)-\varphi(x)\}=v(\bar{x})-\varphi(\bar{x})=0 \quad 
  (\max_{x\in\tilde{\mathcal{O}}\supset\mathcal{O}}\{v(x)-\varphi(x)\}=v(\bar{x})-\varphi(\bar{x})=0),
\end{align} 
then we call $\varphi$ as a super-(sub-) test function for the dependent variable $v$ minimized (maximized) at $\bar{x}$ and denote $\TestFuncMin{\varphi\in\continuousFunction{2}(\tilde{\mathcal{O}})}{\bar{x}\in\mathcal{O}}{v}$  ($\TestFuncMax{\varphi\in\continuousFunction{2}(\tilde{\mathcal{O}})}{\bar{x}\in\mathcal{O}}{v}$) on $\mathcal{O}$.
Or $\TestFuncMin{\varphi}{\bar{x}}{v}$  ($\TestFuncMax{\varphi}{\bar{x}}{v}$) for short, respectively.

つまりこれは次の意味である。
\begin{align}
  (D\varphi(\bar{x}), D^2\varphi(\bar{x}))\in\subsuperjets{v(\bar{x})}{\tilde{\regionO}}, \bar{x}\in\regionO 
\end{align}

なお$\mathcal{O}$が開集合のとき、$\tilde{\mathcal{O}}=\mathcal{O}$にとれるため、以降主に$\mathcal{O}$が開集合のときを示す。

\end{definition}
\end{omitlevel1}

\begin{omitlevel2}
Considering the second-order PDE, $F(x, w(x), Dw(x), D^2 w(x))=0, x\in\mathcal{O}$ for a region $\mathcal{O}\subset \real{}$ for the continuous function 
$F: \mathcal{O}\times  \real{}\times \real{1}\times \real{1} \mapsto \real{}$, which satisfies the following degenerate elliptic condition:
\begin{align}
  \forall x\in\mathcal{O}, r\in\real{}, p\in\real{1}, M, \hat{M}\in\real{1},  M\leq \hat{M} \Rightarrow F(x, r, p, M)\geq F(x, r, p, \hat{M}). \label{def:Ffunction}
\end{align} 
Relating to $'F'$, if there will be no confusion in the context, we will say that differential formula $F(x, w(x), Dw(x), D^2 w(x))$ of the dependent variable $x$ is on $\mathcal{O}$, but the function $F(\cdot, \cdot, \cdot, \cdot)$ is on $\mathcal{O}\times  \real{}\times \real{1}\times \real{1}$.
\end{omitlevel2}

We introduce the following notation.
\begin{notation}[Set of viscosity (super, sub)solutions: \citeN{suzuki2020optimal}] \label{not:solutions}~\\
The set of viscosity solution (viscosity supersolution, viscosity subsolution) to the PDE, $F=0$ on $\mathcal{O}$, respectively is denoted by:
\begin{align}
  \setViscositySolution{F=0}\quad  
  ( \setViscositySuperSolution{F=0}, 
  \setViscositySubSolution{F=0} )
  \text{ on } \mathcal{O}, \text{ respectively.}
\end{align} 
For convenience for the non-differential equation $F=0$, the same notation denotes the classical solution (supersolution, subsolution), i.e., $\setViscositySolution{F=0} (\setViscositySuperSolution{F=0}, \setViscositySubSolution{F=0})$ means $F=0\; (F\geq 0, F\leq 0)$ as in the classical sense, respectively.
\end{notation}

\begin{omitlevel1}
Then we define (discontinuous) viscosity supersolution (viscosity subsolution) as follows, using the above mentioned notations \wrt the test function.
\begin{definition}[Viscosity supersolution (subsolution): \citeN{suzuki2020optimal}] \label{def:solutions}~\\
Let $F$ satisfy (\ref{def:Ffunction}). Locally bounded function $V: \mathcal{O}\mapsto \real{}$ is a viscosity supersolution (subsolution) of $F=0$ on $\mathcal{O}$, respectively, i.e., 
\begin{align}
  &   V\in\setViscositySuperSolution{F({x}, W({x}), DW({x}), D^2 W({x}))=0} \text{ on } \mathcal{O} \nonumber\\
\overset{\text{def}}{\Longleftrightarrow}\quad 
& \forall  {\{{\varphi\in \continuousFunction{2}(\mathcal{O})}, {\bar{x}\in\mathcal{O}}\} \smile  {V}}, F(\bar{x}, V_*(\bar{x}), D\varphi(\bar{x}), D^2 \varphi(\bar{x}))\geq  0,   \\
  &   V\in\setViscositySubSolution{F({x}, W({x}), DW({x}), D^2 W({x}))=0} \text{ on } \mathcal{O} \nonumber\\
\overset{\text{def}}{\Longleftrightarrow}\quad 
& \forall  {\{{\varphi\in \continuousFunction{2}(\mathcal{O})}, {\bar{x}\in\mathcal{O}}\}  \frown  {V}}, F(\bar{x}, V^*(\bar{x}), D\varphi(\bar{x}), D^2 \varphi(\bar{x}))\leq 0,  
\end{align} 
where $V^*(V_*)$ is the upper-(lower-)semicontinuous or u.s.c. (l.s.c.) envelope of the function $V$, respectively.
Moreover, 
\begin{align}
  \setViscositySolution{\cdot}=\setViscositySuperSolution{\cdot}\cap \setViscositySubSolution{\cdot} \text{ for any equation, } '\cdot'. \label{def:supersubsolution}
\end{align} 
\end{definition}
\end{omitlevel1}

\begin{lemma}[Viscosity solution and variational inequality: \citeN{suzuki2020optimal}, Lemma 2.9] \label{prop:VI_supersolution}~\\
Suppose $F=0$ is a second-order partial differential equation and $G=0$ is a non-differential equation  on $\mathcal{O}$, the following holds.
\begin{align}
  &  \setViscositySuperSolution{\min\{F, G\}=0}=\setViscositySuperSolution{F=0}\cap\setViscositySuperSolution{G=0}, \label{eq:VIsuper}\\
  \nonumber\\
  &  \setViscositySubSolution{\min\{F, G\}=0}\supset\setViscositySubSolution{F=0}, \label{eq:VIsubF}\\
  \nonumber\\
  &  \setViscositySubSolution{\min\{F, G\}=0}\supset\setViscositySubSolution{G=0}. \label{eq:VIsubG}
  \intertext{Therefore,}
  &  \setViscositySolution{\min\{F, G\}=0}\supset\setViscositySuperSolution{F=0}\cap\setViscositySolution{G=0}, \label{eq:VIsolution}\\
  \nonumber\\
  &  \setViscositySolution{\min\{F, G\}=0}\supset\setViscositySolution{F=0}\cap\setViscositySuperSolution{G=0}. \label{eq:VIsolution_}
\end{align} 
\end{lemma}

We assume the function $v(z, \xi, n)$ is a viscosity solution to the following variational inequality, which is the obstacle problem with the obstacle function $\valuefuncB{v}{z}{\hat{\xi}}{n-1}-K, \hat{\xi}\in\controlspace(\xi)$.
The following recursive formula corresponds to the iterative switching problem.
\begin{assumption}[Viscosity solution of HJB-variational inequality: \citeN{suzuki2018optimal}] \label{theo:viscosityvn}~\\
  For $n\geq 1$, $\xi\in\positionspace$,
  \begin{align}
    \valuefuncB{v}{z}{\xi}{n}&\in\setViscositySolution{ \min \bigl\{ \delta V -\generator{V}-f(z, \xi), V  -\max_{\hat{\xi}\in\controlspace(\xi)}\{\valuefuncB{v}{z}{\hat{\xi}}{n-1}-K\} \bigr\}=0} \text{ on } \real{},  \label{eq:variationalInequality}\\
    \valuefuncB{v}{z}{\xi}{0}&\in\setViscositySolution{\delta V -\generator{V}-f(z, \xi)=0} \text{ on } \real{},  \label{eq:variationalInequality0}
  \end{align}
  where $\delta>0$ is a discount factor, $K$ is the transaction cost accompanied by each switch and the function $f:\real{}\times\positionspace \mapsto\real{}$ is a running reward function comprising the criterion function $J$ in (\ref{def:valueFunc_criterion}) to be optimized.
\end{assumption}

\begin{assumption}[Comparison principle: ] \label{theo:comparison}
  Let $U_{\xi}$ (respectively $V_{\xi}$), $\xi\in\positionspace$, be a family of u.s.c viscosity subsolutions (respectively l.s.c. viscosity supersolutions) to  equation (\ref{eq:variationalInequality}) with a linear growth condition. Then, $U_{\xi} \leq  V_{\xi}$ holds.
\end{assumption} 

We can formally define the optimal continuation region and the optimal switching region  as (\ref{def:stoppingContinuousRegion}), where for each regime $\xi\in\positionspace$ and the number of remaining switches $n\geq 1$, the whole region $\real{}$ for the state variable $Z\in\real{}$ is divided into two parts, the continuation region $\continuationRegion{n}$ and the switching region $\stoppingRegion{n}$.
For $ \xi\in \positionspace, n\geq 1$ and given $\valuefuncB{v}{z}{\hat{\xi}}{n-1}-K$ and $\hat{\xi}\in\controlspace(\xi)$,
\begin{align}
  &\begin{cases}
     \begin{split}
       \continuationRegion{n} \equiv \{z\in\real{}| \valuefuncB{v}{z}{\xi}{n} &>  \displaystyle\max_{\hat{\xi}\in\controlspace(\xi)}\{\valuefuncB{v}{z}{\hat{\xi}}{n-1}-K\}\},
     \end{split}\\
     \begin{split}
       \stoppingRegion{n} \equiv \{z\in\real{}| \valuefuncB{v}{z}{\xi}{n} & =    \displaystyle\max_{\hat{\xi}\in\controlspace(\xi)}\{\valuefuncB{v}{z}{\hat{\xi}}{n-1}-K\}\},
     \end{split}
  \end{cases} \label{def:stoppingContinuousRegion} 
\end{align}  
where $\controlspace{(\xi)}$ denotes the feasible set of the states (positions) into which the current state $\xi$ can transfer, which is subset of $\positionspace$.

For convenience, we define $\stoppingRegion{0}\equiv\phi(\text{null set}), \; \continuationRegion{0} \equiv \real{}$ and
\begin{align}
  \stoppingRegionA{n}{\xi\hat{\xi}} &\equiv \{z\in\real{}| \valuefuncB{v}{z}{\xi}{n} =    \valuefuncB{v}{z}{{\hat{\xi}}}{n-1}-K\}, & \label{def:stoppingRegioneach}\\
  \stoppingRegionA{0}{\xi\hat{\xi}} &\equiv  \phi, & n=0\nonumber\\
  \continuationRegion{0} &\equiv \real{}, & n=0\nonumber
\end{align}  

\begin{theorem}[Continuity  and uniqueness: \citeN{suzuki2020optimal}] \label{theo:continuity}~\\
    $\vf, \; n\geq 1$ is continuous \wrt $z$ in $\real{}$.
    Moreover, $\vf$ is a unique viscosity solution with linear growth to (\ref{eq:variationalInequality}) in $\real{}$.
\end{theorem}

\begin{lemma}[Viscosity solution $\vf$ to $\delta V -\generator{V}-f=0 ${ on }  $\continuationRegionN$: \citeN{suzuki2020optimal}] \label{prop:solutionOnC}~\\
  \begin{align}
    \valuefuncB{v}{z}{\xi}{n}&\in\setViscositySolution{  \delta V -\generator{V}-f(z, \xi)=0} \text{ on } \continuationRegionN,\; n\geq 1
  \end{align}
\end{lemma}

\begin{lemma}[$\vf$ classically solves   $\delta V -\generator{V}-f=0 ${ on }  $\continuationRegionN$: \citeN{suzuki2020optimal}]
  \label{theo:classicalPart}~\\
$\vf, \; n\geq 1$ is a classical solution to the following differential equation on $\continuationRegionN$ with the boundary condition,  and it is unique.
\begin{align}
  -\generator{V}+\delta V &=f  \text{ on } \continuationRegionN,\\
  V&=v  \text{ in } \partial\continuationRegionN,\\
  \text{moreover, }\; \vf \in C^2(\continuationRegionN) \label{eq:ODE} 
\end{align}
\end{lemma}

From Lemma \ref{theo:classicalPart} and (\ref{def:stoppingContinuousRegion}), the following corollary holds.
\begin{corollary} \label{eq:valuefunctionallregionN}
    {For $n\geq 1$, the value function $v$ (classically)  solves either of the following equations depending on $z\in\real{}=\continuationRegionN\oplus\stoppingRegionN$:} 
  \begin{align}
    \begin{cases}
      \delta v -\generator{v}-f=0, & z \in \continuationRegionN, \\
      v  -\max_{\hat{\xi}\in\controlspace(\xi)}\{\valuefuncB{v}{z}{\hat{\xi}}{n-1}-K\}   =0,  & z \in \stoppingRegionN. \label{eq:DEclassical}
    \end{cases}
  \end{align}
\end{corollary}
\def\plainA#1{  {{#1}}^{\xi\hat{\xi}}_n  }
\def\Ominus{{\regionO}^{-}}
\def\Oplus{{\regionO}^{+}}
\def\CC#1#2{C^{#1, #2}_n}
\def\C#1{C^{#1}_n}

\subsection{Simultaneous Switchings and Structure of the Switching Regions}\label{sec:struct_switch}


 If the exact $m$time(s) consecutive simultaneous nonrecurrent switchings $\xi^{(0)}\mapsto\cdots\mapsto\xi^{(m)}$ to the stable continuation state (regime) $\continuationRegionA{n-m}{{\xi}^{(m)}}$  is optimal,  we define the region $\stoppingRegionA{n}{\xi^{(0)}\cdots{\xi}^{(p)}}$ for the optimal intermediate series of $p\leq m$ simultaneous  consecutive switchings $\xi^{(0)}\mapsto\cdots \mapsto{\xi}^{(p)}$ out of its optimal full series of $m$ consecutive switchings for $p=1, \cdots m$ as follows:
\begin{align}
  \stoppingRegionA{n}{\xi^{(0)}\cdots{\xi}^{(p)}}\equiv \bigcap_{\begin{subarray}{l} {\xi}^{(q+1)}\in\controlspace(\xi^{(q)})\\ q=0, \cdots, p-1 \end{subarray} } \stoppingRegionA{n-q}{\xi^{(q)}{\xi}^{(q+1)}}. \label{eq:simulSwitch}
\end{align}
For convenience, for $p=0$ we let $\stoppingRegionA{n}{\xi^{(0)}\cdots{\xi}^{(p)}}=\real{}$, and not $\stoppingRegionA{n}{\xi^{(0)}}$ as in (\ref{def:stoppingContinuousRegion}).
 
We define the following disjoint stable state region  $\continuationAfterSwitchingRegionA{n,m}{\xi^{(0)}\cdots{\xi}^{(m)}}\subset\real{}$, where the exact $m-$time(s) series of simultaneous nonrecurrent switchings ${\xi^{(0)}\mapsto\xi^{(1)}\mapsto \cdots\mapsto{\xi}^{(m)}}$ for ${\xi}^{(p)}\in\controlspace(\xi^{(p-1)}), p=1, \cdots, m$ to the stable continuation state $\continuationRegionA{n-m}{{\xi}^{(m)}}$, is the optimal switching strategy at $z\in\continuationAfterSwitchingRegionA{n,m}{\xi^{(0)}\cdots{\xi}^{(m)}}$:  for ${\xi}^{(p)}\in\controlspace(\xi^{(p-1)}), p=1, \cdots, m$ for $m\geq 0$, 
\begin{align}
  \continuationAfterSwitchingRegionA{n,m}{\xi^{(0)}\cdots{\xi}^{(m)}}&\equiv   \stoppingRegionA{n}{\xi^{(0)}\cdots{\xi}^{(m)}} \cap \continuationRegionA{n-m}{{\xi}^{(m)}} \qquad \cdots (m\geq 1), \label{eq:relationOSC}
\end{align}
where in our rule of notation, if the series $\xi^{(0)}\cdots{\xi}^{(m)}$ of the states for switchings is empty, the corresponding simultaneous switching region $\stoppingRegionA{n}{\xi^{(0)}\cdots{\xi}^{(m)}}$ should also be empty. That is,
\begin{align}
  \continuationAfterSwitchingRegionA{n,0}{\xi^{(0)}}&\equiv \continuationRegionA{n}{{\xi}^{(0)}} \qquad \cdots (m=0).
\end{align}

\begin{omitlevel3}
The switching regions are expressed to be divided into the series of the continuation state sets $\continuationAfterSwitchingRegionA{n,m}{\xi^{(0)}\cdots{\xi}^{(m)}}$ to switch to as follows.
\begin{align}  
  \real{}&=\bigoplus_{\begin{subarray}{l} {\xi}^{(p)}\in\controlspace(\xi^{(p-1)})\\ p=1, \cdots, m\\ m=0,1,\cdots \end{subarray}}    \continuationAfterSwitchingRegionA{n,m}{\xi^{(0)}\cdots{\xi}^{(m)}} \text{ for } \forall\xi^{(0)}\in \positionspace   \label{eq:seemlessR}\\
  \stoppingRegionA{n}{\xi^{(0)}}&=\bigoplus_{\begin{subarray}{l} {\xi}^{(p)}\in\controlspace(\xi^{(p-1)})\\ p=1, \cdots, m\\ m=1,2,\cdots \end{subarray}}    \continuationAfterSwitchingRegionA{n,m}{\xi^{(0)}\cdots{\xi}^{(m)}} \text{ for } \xi^{(0)}\in \positionspace  \label{eq:seemlessStoppingRegion}\\
  \stoppingRegionA{n}{\xi^{(0)}\xi^{(1)}}&=\bigoplus_{\begin{subarray}{l} {\xi}^{(p)}\in\controlspace(\xi^{(p-1)})\\ p=2, \cdots, m\\ m=1,2,\cdots \end{subarray}}    \continuationAfterSwitchingRegionA{n,m}{\xi^{(0)}\cdots{\xi}^{(m)}} \text{ for } {\xi}^{(1)}\in\controlspace(\xi^{(0)}), \xi^{(0)}\in \positionspace,  \label{eq:seemlessStoppingRegion2}
\end{align} 
where when the set $p=2, \cdots, m$ of indices of the nonrecurrent states is empty, corresponding series of sets ${\xi}^{(p)}\in\controlspace(\xi^{(p-1)})$ are also empty and  the only set $\continuationAfterSwitchingRegionA{n,m}{\xi^{(0)}\cdots{\xi}^{(m)}}$ remains, i.e., $\bigoplus_{\begin{subarray}{l} \{\phi\}\\ m=1 \end{subarray}} \continuationAfterSwitchingRegionA{n,m}{\xi^{(0)}\cdots{\xi}^{(m)}} = \continuationAfterSwitchingRegionA{n,1}{\xi^{(0)}{\xi}^{(1)}}$.
If $m=0$ and the set of indices $p=\cdots, m$ is empty, $\continuationAfterSwitchingRegionA{n,0}{\xi^{(0)}}=\continuationRegionA{n}{{\xi}^{(0)}}$ is included.
\end{omitlevel3}

\def\Ominus{{\regionO}^{-}}
\def\Oplus{{\regionO}^{+}}

\begin{lemma}[Special test functions on the optimal switching regions: \citeN{suzuki2020optimal}] \label{prop:test_funct_simultaneous}~\\
For each number $m=0, 1, \cdots$  of consecutive simultaneous switchings before continuation, the function: 
\begin{align}
  \varphi^{\xi^{(m)}}_{n, m}(z)\equiv\valuefuncB{v}{z}{\xi^{(m)}}{n-m} - mK \in\continuousFunction{2}(\continuationRegionA{n-m}{{\xi}^{(m)}})  \text{ (from (\ref{eq:ODE})) defined on the open set } \continuationRegionA{n-m}{{\xi}^{(m)}}\supset\continuationAfterSwitchingRegionA{n,m}{\xi^{(0)}\cdots{\xi}^{(m)}}
\end{align}
is a super test function for the value function $\valuefuncB{v}{z}{\xi}{n}$  minimized at any $\bar{z}\in\continuationAfterSwitchingRegionA{n,m}{\xi^{(0)}\cdots{\xi}^{(m)}}\subset\continuationRegionA{n-m}{{\xi}^{(m)}}$, i.e.,
\begin{align}
  \TestFuncMin{{\varphi^{\xi^{(m)}}_{n, m}(z)} \in \continuousFunction{2}(\continuationRegionA{n-m}{{\xi}^{(m)}})}{\forall\bar{z}\in\continuationAfterSwitchingRegionA{n,m}{\xi^{(0)}\cdots{\xi}^{(m)}}}{\valuefuncB{v}{z}{\xi}{n}} \text{ on } \continuationAfterSwitchingRegionA{n,m}{\xi^{(0)}\cdots{\xi}^{(m)}}.\label{eq:specialTestFunc}
\end{align}
Notation relating to the test function is based upon \citeN{suzuki2020optimal}.

Moreover,  on $\continuationAfterSwitchingRegionA{n,m}{\xi^{(0)}\cdots{\xi}^{(m)}}$  for ${\xi}^{(p+1)}\in\controlspace(\xi^{(p)}),\quad \xi^{(0)}\in\positionspace$,
\begin{align}
  \varphi^{\xi^{(0)}}_{n, 0}(z)&= \varphi^{\xi^{(p)}}_{n, p}(z)\quad p=0, \cdots, m.   \label{eq:special_testFunction_p}
\end{align}
\end{lemma}

\section{Piecewise Classical Viscosity Solution}\label{sec:Piecewise classical viscosity solution}
In \citeN{suzuki2020optimal} we study the regularity condition of the value function. 
Apart from the value function itself, we will consider the regularity condition for the viscosity solution solved by the value function.
The viscosity solution is composed of the functions on the continuation regions and those of the switching regions. According to   Lemma \ref{theo:classicalPart}, the function on a continuation region is assumed to solve the PDE classically, that is, the function is in $\continuousFunction{2}$ on the region, while the function on the switching region is a combination of some given $\continuousFunction{2}$-functions. We are interested in the regularity conditions of the connection points between those functions on the continuation regions and those of the switching regions in the problem which allow simultaneous multiple switchings.

We would like to construct the viscosity solution on $\real{}$  from the series of the piecewise $\continuousFunction{2}$-classical solutions.
First, in Theorem \ref{theo:C1 cond at connection} we prove that $\continuousFunction{1}$-condition (smooth pasting condition) at the connection points is a necessary condition for the connected function to be a viscosity solution.
Then in Theorem \ref{theo:piecewiseClassicalViscositySolution} a sufficient condition for the connected function to be the viscosity solution on $\real{}$ is provided.
Based on the condition Algorithm \ref{algo:system_freeBoundary} provides the algorithm to calculate the concrete viscosity solution in a computer.
In Section \ref{sec:example_model} the algorithm is applied to the example problem studied in \citeN{suzuki2020optimal} and calculate the explicit solution.

\begin{theorem}[Smooth pasting condition ($\continuousFunction{1}$-condition) of the solution at the free boundary] \label{theo:C1 cond at connection}~\\
  For each $\xi\in\positionspace$, $\valuefuncB{v}{z}{\hat{\xi}}{n-1}$ for  ${\hat{\xi}\in\controlspace(\xi)}$ are given value functions (\ref{def:valueFunc_criterion}) and the non-degenerate elliptic PDE: $F(z, V(z), DV(z), D^2V(z))=0$ has a classical particular solution $u_0$ on $\overline{\continuationRegion{n}}$.
  Suppose that  
  \begin{align}
    u(z)&\in\setViscositySolution{ \min \bigl\{ F(z, V(z), DV(z), D^2V(z)), V  -\max_{\hat{\xi}\in\controlspace(\xi)}\varphi^{\hat{\xi}}_{n, 1}(z) \bigr\}=0} \text{ on } \real{},  \label{rel:u}
  \end{align}
  and $u(z)=u_0(z)$ on $\overline{\continuationRegion{n}}$.
  Then the function $u$ is of class $\continuousFunction{1}$ at $\bar{z}\in\partial\continuationRegionN$, where 
 \begin{align}
   &\begin{cases} 
     \begin{split}
       \continuationRegion{n} = \{z\in\real{}| u(z) &>  \displaystyle\max_{\hat{\xi}\in\controlspace(\xi)}\varphi^{\hat{\xi}}_{n, 1}(z)\}, 
     \end{split}\\
     \begin{split}
       \stoppingRegion{n} = \{z\in\real{}| u(z) & =     \displaystyle\max_{\hat{\xi}\in\controlspace(\xi)}\varphi^{\hat{\xi}}_{n, 1}(z)\},
     \end{split}
    \end{cases} \label{def:stoppingContinuousRegion_u} 
 \end{align}
\end{theorem}

\begin{proof}~\\ 
  For $m\geq 1$, if the exact $m$ time(s) consecutive simultaneous nonrecurrent switching(s): $\xi\mapsto\xi^{(1)}\mapsto\cdots\mapsto\xi^{(m)}, m\geq 1$ to the stable continuation state (regime) $\continuationRegionA{n-m}{{\xi}^{(m)}}$  is the optimal at $z=\bar{z}$, then
  from (\ref{eq:relationOSC})
  \begin{align}
    \bar{z}\in\partial\continuationRegionN\subset\continuationAfterSwitchingRegionA{n,m}{\xi\xi^{(1)}\cdots{\xi}^{(m)}}&=\stoppingRegionA{n}{\xi\xi^{(1)}\cdots{\xi}^{(m)}} \cap \continuationRegionA{n-m}{{\xi}^{(m)}}. \label{in:bar z in}
  \end{align}
  from the definition (\ref{def:stoppingContinuousRegion_u}) and Lemma \ref{prop:test_funct_simultaneous}, for $m\geq 1$,  
  \begin{align}
    &\begin{cases}
       u(z)&>\varphi^{\xi^{(1)}}_{n, 1}(z)=\cdots =\varphi^{\xi^{(m)}}_{n, m}(z) \text{ on } \continuationRegionN,\\ 
       u(z)&=\varphi^{\xi^{(1)}}_{n, 1}(z)=\cdots =\varphi^{\xi^{(m)}}_{n, m}(z) \text{ on }\continuationAfterSwitchingRegionA{n,m}{\xi\xi^{(1)}\cdots{\xi}^{(m)}} \subset \continuationRegionA{n-m}{{\xi}^{(m)}} \backslash \continuationRegionN \ni \bar{z},
     \end{cases} \label{rel:u_z}
  \end{align}

  \begin{align}
    &\begin{cases}
       \dfrac{u(z)-u(\bar{z})}{(sgn)(z-\bar{z})}&>\dfrac{\varphi^{\xi^{(1)}}_{n, 1}(z)-\varphi^{\xi^{(1)}}_{n, 1}(\bar{z})}{(sgn)(z-\bar{z})}\text{ on } \continuationRegionN,\\ 
       \dfrac{u(z)-u(\bar{z})}{(sgn)(z-\bar{z})}&=\dfrac{\varphi^{\xi^{(1)}}_{n, 1}(z)-\varphi^{\xi^{(1)}}_{n, 1}(\bar{z})}{(sgn)(z-\bar{z})}\text{ on } \continuationAfterSwitchingRegionA{n,m}{\xi\xi^{(1)}\cdots{\xi}^{(m)}} \subset \continuationRegionA{n-m}{{\xi}^{(m)}} \backslash \continuationRegionN, 
     \end{cases} \label{rel:u_z2}
  \end{align}
  where $sgn$ is the  $\continuationRegionN$-side indicator from $\bar{z}$: 
  \begin{align}
    &sgn=
    \begin{cases}
      	1 & \cdots  [\bar{z}, +\infty )\cap \continuationRegionN\cap B_{\varepsilon}(\bar{z})\ne\phi, \forall \varepsilon>0,\\
      	-1 & \cdots  (-\infty, \bar{z}]\cap \continuationRegionN\cap B_{\varepsilon}(\bar{z})\ne\phi, \forall \varepsilon>0,
    \end{cases} \label{def:sgn}
  \end{align}

  At this stage related to the 1-st order derivative of $u$ at $z=\bar{z}$, there exist:
  \begin{align}
    \begin{cases}
      & \text{the left and right derivative of } u_0\in\continuousFunction{2}(\overline{\continuationRegion{n}})  \text{ on } \overline{\continuationRegion{n}},\\ 
      & \text{the }\continuationRegionN\text{-side (one-side) derivative } u_C'(\bar{z})\equiv u_0'(\bar{z}) \text{ of }  u(z),\\
      & \text{the left and right derivative of } \varphi^{\xi^{(m)}}_{n, m}\in\continuousFunction{2}(\continuationRegionA{n-m}{{\xi}^{(m)}}) \text{ on } \continuationRegionA{n-m}{{\xi}^{(m)}},\\
      & \text{the }\stoppingRegionA{n}{\xi\xi^{(1)}}\text{-side (one-side) derivative } u_S'(\bar{z})\equiv \varphi^{\xi^{(m)}}_{n, m}{'}(\bar{z}) \text{ of }  u(z),
    \end{cases} \label{rel:C2cond}
  \end{align}
  but we do not yet have the derivative $u'(\bar{z})$ of $u$ at $z=\bar{z}\in\continuationRegionN\cup\stoppingRegionA{n}{\xi\xi^{(1)}}$.
  Taking the one-side limit $z\to\bar{z}+(sgn)0$ of each of (\ref{rel:u_z2}), 
  \begin{align}
    (sgn)u_0'(\bar{z})=(sgn)u_C'(\bar{z})\geq(sgn)\varphi^{\xi^{(m)}}_{n, m}{'}(\bar{z})=(sgn)u_S'(\bar{z}).
  \end{align}
  We will consider only the case of $sgn=-1$. The other case will be similar. In this case
  \begin{align}
    u_C'(\bar{z})\leq u_S'(\bar{z}).
  \end{align}
  We will contradict the case 
  \begin{align}
    u_C'(\bar{z})<u_S'(\bar{z}). \label{ineq:contradiction_u'}
  \end{align}
  Assuming (\ref{ineq:contradiction_u'}) the semijets of $u$ at $\bar{z}$ on the neighborhood $B_\varepsilon(\bar{z})$ of $\bar{z}$ are as follows:
  \begin{align}
    \begin{cases}
      \subjets{u(\bar{z})}{B_\varepsilon(\bar{z})}=\{(u_C'(\bar{z}), u_S'(\bar{z}))\times\real{}\}\cup\{\{u_C'(\bar{z}), u_S'(\bar{z})\}\times(-\infty, 0]\},\\
\superjets{u(\bar{z})}{B_\varepsilon(\bar{z})}=\phi.
    \end{cases} \label{semijets_uz}
  \end{align}
  In order for $u$ to be a supersolution to  the variational inequality in (\ref{rel:u}) at $z=\bar{z}$, 
  for any $(p, +\infty)\in \subjets{u(\bar{z})}{B_\varepsilon(\bar{z})}$, where $p\in(u_C'(\bar{z}), u_S'(\bar{z}))$, the inequality $F(\bar{z}, u(\bar{z}), p, \infty)\geq 0$ should be satisfied, which yields contradiction, due to the nondegenerate ellipticity and properness of $F$.
  Therefore, 
  \begin{align}
    u_C'(\bar{z})=u_S'(\bar{z}), \label{eq:admit_derivative_u}
  \end{align}
  which means $u$ is of class $\continuousFunction{1}$ at $\bar{z}$.
\end{proof} 

We define the following reference set on $\real{}$ for $m$ simultaneous multiple-regime switches from $\xi$ to $\hat{\xi}$.
For $\xi, \hat{\xi}\in\positionspace, m\geq 1$,
\begin{align}
  \stoppingRegionCandidate{\xi\hat{\xi}}{m}\equiv\{z|f(z, \hat{\xi})-f(z, \xi)\geq m\delta K\}. \label{def:candidateRegion}
\end{align}

\begin{lemma}[Reference regions for switching: \citeN{suzuki2020optimal}] \label{prop:candidateRegion}~\\
If the sequence of the full consecutive simultaneous switchings: $\xi^{(0)}\mapsto\cdots\mapsto{\xi}^{(m)}$ for ${\xi}^{(p+1)}\in\controlspace(\xi^{(p)}), p=0, \cdots, m-1, \xi^{(0)}\in\positionspace, m\geq 1$, is the optimal switching strategy, where the last regime (position) ${\xi}^{(m)}$ is the continuation state at $z$, then 
\begin{align}
  \stoppingRegionCandidate{\xi^{(0)}\xi^{(m)}}{m}\supset \continuationAfterSwitchingRegionA{n,m}{\xi^{(0)}\cdots{\xi}^{(m)}} \label{rel:QO}
\end{align}
\end{lemma}
Note that the above proposition assert that the reference region $\stoppingRegionCandidate{\xi^{(0)}\xi^{(1)}}{1}$ should include only $\continuationAfterSwitchingRegionA{n,1}{\xi^{(0)}{\xi}^{(1)}}\subset\stoppingRegionA{n}{\xi^{(0)}\xi^{(1)}}$ and not necessarily include whole the set $\stoppingRegionA{n}{\xi^{(0)}\xi^{(1)}}$, as the optimal switching problems without simultaneous switchings. Therefore we introduce the following assumption.

\begin{assumption}[Structure of switching regions: \citeN{suzuki2020optimal}] \label{assumption:structureSwitchingRegion}~\\
(i) For each $\xi^{(0)}\in\positionspace$, the sets $\stoppingRegionCandidate{\xi^{(0)}\xi^{(1)}}{1},\quad {\xi}^{(1)}\in\controlspace(\xi^{(0)})$ are mutually disjoint.
As a special case, if the set $\controlspace(\xi^{(0)})$ is a singleton, this condition is immediately satisfied.\\
(ii) $\stoppingRegionCandidate{\xi^{(0)}\xi^{(1)}}{1} \supset \stoppingRegionA{n}{\xi^{(0)}\xi^{(1)}}$ for $\forall {\xi}^{(1)}\in\controlspace(\xi^{(0)}), \forall \xi^{(0)}\in\positionspace$.
\end{assumption}
Under the above assumption we can make each set $\stoppingRegionA{n}{\xi\hat{\xi}}\subset\stoppingRegionCandidate{\xi\hat{\xi}}{1}$ for $\xi\in\positionspace$ mutually disjoint.  Therefor the isolated or crossing boundary intersection points in $\stoppingRegionA{n}{\xi\hat{\xi}} \cap \stoppingRegionA{n}{\xi\bar{\xi}}, \hat{\xi}\ne\bar{\xi},\; \hat{\xi},\bar{\xi}\in\controlspace(\xi)$ do not exist, i.e., for each $\xi$,  
\begin{align} 
  \stoppingRegionA{n}{\xi\hat{\xi}} \cap \stoppingRegionA{n}{\xi\bar{\xi}}=\phi, \hat{\xi}\ne\bar{\xi},\; \hat{\xi},\bar{\xi}\in\controlspace(\xi). \label{eq:empty_intersectSS}
\end{align}

Therefore $\stoppingRegionA{n}{\xi} =\bigoplus_{\hat{\xi}\in\controlspace(\xi)} \stoppingRegionA{n}{\xi\hat{\xi}}$, where $'\oplus'$ operator operates as direct sum (disjoint union).
In this case for each $\xi\in\positionspace$,
\begin{align} 
  \real{}=\bigg(\bigoplus_{\hat{\xi}\in\controlspace(\xi)} \stoppingRegionA{n}{\xi\hat{\xi}}\bigg)\oplus\continuationRegionA{n}{\xi},  \label{eq:decompR}
\end{align}

\begin{corollary}[Decomposition of $\real{}$ into the regions: \citeN{suzuki2020optimal}] \label{coro:decompR}~\\
Under Assumption \ref{assumption:structureSwitchingRegion}, the following decomposition of $\real{}$ holds for each $\xi^{(0)}\in\positionspace$ according to (\ref{eq:decompR}).
\begin{align} 
  \real{}=\bigg(\bigoplus_{{\xi}^{(1)}\in\controlspace(\xi^{(0)})} \stoppingRegionCandidate{\xi^{(0)}\xi^{(1)}}{1}\bigg)\cup\continuationRegionA{n}{{\xi}^{(0)}}, \label{eq:unionSetsR}
\end{align}
\end{corollary}

\begin{omitlevel1}
Considering the problems involving the  simultaneous switching, to verify Assumption \ref{assumption:structureSwitchingRegion} (ii) might be difficult in general.
We provide a sufficient condition for Assumption \ref{assumption:structureSwitchingRegion} (ii) in the following.
\begin{proposition}[A sufficient condition for Assumption \ref{assumption:structureSwitchingRegion} (ii) : \citeN{suzuki2020optimal}] \label{prop:sufficientCond}~\\
  If for each  ${\xi_1}\in\controlspace(\xi_0), \xi_0\in\positionspace$,
  \begin{align} 
    \stoppingRegionCandidate{\xi_0\xi_1}{1}\supset  \bigcup_{\bigg\{\begin{subarray}{l}  {\xi}_m\in\controlspace(\xi_{m-1})\\ \xi_{m-1}\in\positionspace \\ \xi_{m-1}\ne\xi_{0} \\ \xi_{m}\ne\xi_{0}, \xi_{1}   \end{subarray}\bigg\}} \stoppingRegionCandidate{\xi_{m-1}\xi_{m}}{1},
  \end{align}
  then Assumption \ref{assumption:structureSwitchingRegion} (ii) is satisfied.
\end{proposition}

For each  ${\xi_{1}}\in\controlspace(\xi_{0}), \xi_{0}\in\positionspace$, the actual set 
$\stoppingRegionCandidate{\xi_{0}\xi_{1}}{1}$ might consist more than two regions. In this case it is further decomposed into mutually disjoint separated single regions as follows:
\begin{align} 
  \stoppingRegionCandidate{\xi_{0}\xi_{1}}{1}=\bigoplus_{\alpha_i\in\Lambda}  \stoppingRegionCandidate{\xi_{0}\xi_{1}\{\alpha_i\}}{1},
\end{align}
where each $\stoppingRegionCandidate{\xi_{0}\xi_{1}\{\alpha_i\}}{1}, \alpha_i\in\Lambda$ is a single connected region.
So we group those series of regions by the index $\{\alpha_i\}$ as follows: for the optimal series of simultaneous switches ${\xi^{(0)}\cdots{\xi}^{(m)}}$ and for each $\alpha_i\in\Lambda$,
\begin{align} 
 \stoppingRegionA{n}{\xi^{(0)}\xi^{(1)}\{\alpha_i\}}\equiv \stoppingRegionA{n}{\xi^{(0)}\xi^{(1)}}\cap  \stoppingRegionCandidate{\xi^{(0)}\xi^{(1)}\{\alpha_i\}}{1} \label{def:Salpha} \\
 \continuationAfterSwitchingRegionA{n,m}{\xi^{(0)}\cdots{\xi}^{(m)}\{\alpha_i\}}\equiv\continuationAfterSwitchingRegionA{n,m}{\xi^{(0)}\cdots{\xi}^{(m)}} \cap  \stoppingRegionCandidate{\xi^{(0)}\xi^{(1)}\{\alpha_i\}}{1}.
\end{align} 
Taking the union of the both sides of the equation (\ref{def:Salpha}) over the index $\{\alpha_i\}$, 
\begin{align} 
 \bigoplus_{\alpha_i\in\Lambda}\stoppingRegionA{n}{\xi^{(0)}\xi^{(1)}\{\alpha_i\}}&=\stoppingRegionA{n}{\xi^{(0)}\xi^{(1)}}\cap  \big(\bigoplus_{\alpha_i\in\Lambda}\stoppingRegionCandidate{\xi^{(0)}\xi^{(1)}\{\alpha_i\}}{1}\big)\\
 &=\stoppingRegionA{n}{\xi^{(0)}\xi^{(1)}}\cap  \stoppingRegionCandidate{\xi^{(0)}\xi^{(1)}}{1}.
\end{align}
Therefore under Assumption \ref{assumption:structureSwitchingRegion} (ii), the following corollary holds.
\begin{corollary}[Separated switchings regions: \citeN{suzuki2020optimal}] \label{prop:separatedSwitchingRegions}~\\
Under Assumption \ref{assumption:structureSwitchingRegion} (ii),
\begin{align} 
 \bigoplus_{\alpha_i\in\Lambda}\stoppingRegionA{n}{\xi^{(0)}\xi^{(1)}\{\alpha_i\}}&=\stoppingRegionA{n}{\xi^{(0)}\xi^{(1)}} \label{eq:separatedSwitchingRegions}
\end{align}
\end{corollary}

\begin{proposition}[Connected region in the reference region: \citeN{suzuki2020optimal}] \label{prop:connectedSwitchingRegion}~\\
  Defining the connected region as $\overline{\stoppingRegionA{n}{\xi^{(0)}\xi^{(1)}\{\alpha_i\}}} \equiv [\inf {\stoppingRegionA{n}{\xi^{(0)}\xi^{(1)}\{\alpha_i\}}}, \sup {\stoppingRegionA{n}{\xi^{(0)}\xi^{(1)}\{\alpha_i\}}} ]$, the following relation holds: 
\begin{align}
  \overline{\stoppingRegionA{n}{\xi^{(0)}\xi^{(1)}\{\alpha_i\}}}   \subset \stoppingRegionCandidate{\xi^{(0)}\xi^{(1)}\{\alpha_i\}}{1}.
\end{align} 
\end{proposition}

We introduce another HJB-variational inequality, which is slightly different from (\ref{eq:variationalInequality}):
\begin{align}
  \min \bigl\{ \delta V -\generator{V}-f(z, \xi), V  -\{\valuefuncB{v}{z}{\hat{\xi}}{n-1}-K\} \bigr\}=0 \text{ on } \real{},  \label{eq:variationalInequality2}
\end{align}
for any $\hat{\xi}\in\controlspace(\xi), \xi\in\positionspace$, and $\valuefuncB{v}{z}{\hat{\xi}}{n-1}, 1\leq n\leq +\infty$.

\begin{proposition}[Another HJB-variational inequality: \citeN{suzuki2020optimal}] \label{prop:anotherVI}~\\
  Under Assumption \ref{assumption:structureSwitchingRegion}, 
  for any ${\xi^{(1)}}\in\controlspace(\xi^{(0)}), \xi^{(0)}\in\positionspace$, and given $\varphi^{\xi^{(1)}}_{n, 1}(z)=\valuefuncB{v}{z}{\xi^{(1)}}{n-1},\; n\geq 1$,  
  \begin{align}
    \valuefuncB{v}{z}{\xi^{(0)}}{n}&\in\setViscositySolution{ \min \bigl\{ \delta V -\generator{V}-f(z, \xi^{(0)}), V  - \varphi^{\xi^{(1)}}_{n, 1}(z) \bigr\}=0} \text{ on } \overline{\stoppingRegionA{n}{\xi^{(0)}\xi^{(1)}\{\alpha_i\}}},  \label{eq:variationalInequality3}
  \end{align} 
for each     $\alpha_i\in\Lambda$.
\end{proposition}

\begin{proposition}[Another HJB-variational inequality and the special test function: \citeN{suzuki2020optimal}] \label{prop:specialTestFunc}~\\
  For any ${\xi^{(1)}}\in\controlspace(\xi^{(0)}), \xi^{(0)}\in\positionspace$, and given $\varphi^{\xi^{(1)}}_{n, 1}(z)=\valuefuncB{v}{z}{\xi^{(1)}}{n-1}, 1\leq n\leq +\infty$,  
  \begin{align}
    \varphi^{\xi^{(1)}}_{n, 1}(z)&\in\setViscositySolution{ \min \bigl\{ \delta V -\generator{V}-f(z, \xi^{(0)}), V  - \varphi^{\xi^{(1)}}_{n, 1}(z) \bigr\}=0} \text{ on } \overline{\stoppingRegionA{n}{\xi^{(0)}\xi^{(1)}\{\alpha_i\}}}  \label{eq:variationalInequality4}
  \end{align} 
  for each     $\alpha_i\in\Lambda$.
\end{proposition}

\begin{theorem}[Connected switching regions: \citeN{suzuki2020optimal}] \label{prop:connectedRegion}~\\
  Under Assumption \ref{assumption:structureSwitchingRegion}, 
  for each index $\alpha_i\in\Lambda$, the set ${\stoppingRegionA{n}{\xi^{(0)}\xi^{(1)}\{\alpha_i\}}}$ is a single connected region, i.e., 
  \begin{align}
    &\overline{\stoppingRegionA{n}{\xi^{(0)}\xi^{(1)}\{\alpha_i\}}}={\stoppingRegionA{n}{\xi^{(0)}\xi^{(1)}\{\alpha_i\}}}
  \end{align} 
  for each index $\alpha_i\in\Lambda$ and ${\xi^{(1)}}\in\controlspace(\xi^{(0)}), \xi^{(0)}\in\positionspace$.
\end{theorem}
\end{omitlevel1}

In (\ref{def:stoppingContinuousRegion}) we have only formally defined the switching regions and the continuation regions 
whose structure is identified in the following summarizing remark. 
\begin{corollary}[Structure of switching regions: \citeN{suzuki2020optimal}] \label{coro:structureSwitchingRegion}~\\
  Under Assumption \ref{assumption:structureSwitchingRegion},
  for each $\xi^{(0)}\in\positionspace$,
  ${\stoppingRegionA{n}{\xi^{(0)}\xi^{(1)}\{\alpha_i\}}}\subset\stoppingRegionCandidate{\xi^{(0)}\xi^{(1)}\{\alpha_i\}}{1}$
  is at most only one connected region. 
  $\bigoplus_{\alpha_i\in\Lambda}\stoppingRegionA{n}{\xi^{(0)}\xi^{(1)}\{\alpha_i\}}=\stoppingRegionA{n}{\xi^{(0)}\xi^{(1)}}$, 
  where each $\stoppingRegionCandidate{\xi^{(0)}\xi^{(1)}\{\alpha_i\}}{1}, {\xi^{(1)}}\in\controlspace(\xi^{(0)}), \alpha_i\in\Lambda$ is disjoint.
\end{corollary}

\begin{omitlevel3}
In order to identify the structure of the switching regions, another useful criterion is the non-empty condition, which depends on the uniqueness property of the viscosity solution on an unbounded domain when the solution grows linearly.
For the case of our problem, we will prove the following non-empty condition, when the solution is quadratic form on unbounded domain.
\begin{lemma}[Non-empty condition: \citeN{suzuki2020optimal}] \label{theo:nonempty}
  For $\xi\in\positionspace$, 
  \begin{align}
    \exists z\in\real{}, \exists \hat{\xi}\in\controlspace(\xi) s.t., v(z, \xi, 0)\leq v(z, \hat{\xi}, 0)-K \Longrightarrow \forall n\geq 1, \stoppingRegionA{n}{\xi}\ne \phi  \label{cond:nonEmpty}
  \end{align}
\end{lemma} 
\end{omitlevel3}

\def\CC#1#2{C^{#1,#2}_{n}}
\def\CCn#1#2#3{C^{#1,#2}_{#3}}

\begin{omitlevel1}
\subsection{Identifying Switching Regions} 
ファイナンスの文脈の中で最適戦略を取り扱う以上、
具体的な数値として粘性解を求めることが最終目標である。
微分方程式:$F(\bar{x}, V_*(\bar{x}), D\varphi(\bar{x}), D^2 \varphi(\bar{x}))=0$の粘性解の定義はDefinition \ref{def:solutions}で与えられるが、直接この定義から解を決定するためには、以下のステップを要する。
\begin{enumerate}
\item 候補となる方程式の解$V(x)$を求める	    	
\item $V$の粘性優解性を示す\\
  $\varphi\in \continuousFunction{2}(\tilde{\mathcal{O}})$及び$\bar{x}\in\mathcal{O}$を満たすあらゆる$V$の試験関数$\varphi$ ($\TestFuncMin{\forall\varphi\in\continuousFunction{2}(\tilde{\mathcal{O}})}{\forall\bar{x}\in\mathcal{O}}{V}$)に対して、$F(\bar{x}, V_*(\bar{x}), D\varphi(\bar{x}), D^2 \varphi(\bar{x}))\geq  0$であることを示す
\item 同様に$V$の粘性劣解性を示す   	
\end{enumerate}
しかし、計算機を用いて上記ステップ2の全ての試験関数$\varphi$:$\TestFuncMin{\varphi\in\continuousFunction{2}(\tilde{\mathcal{O}})}{\bar{x}\in\mathcal{O}}{V}$に対し、この定義に忠実に検証を行うことは不可能であろう。
$'\forall'$という演算を通常の計算機上の数値計算で実行できない。
他にできることといえば、Lemma \ref{theo:classicalPart}を用いることにより、継続領域$\continuationRegionN$上という限定的な定義域において粘性解を古典解として求めることくらいである。

また、Corollary \ref{eq:valuefunctionallregionN}を用いて、継続領域$\continuationRegionN$上で求めた解とスイッチ領域$\stoppingRegionN$上の関数との連結部分を
Theorem \ref{theo:smoothfitvn} を用いて貼り合わせれば、必要とされる条件のいくつかを満たすような関数は得られるかもしれない。こうして得られた関数は求めるべき値関数(或は変分不等式の粘性解)であるためのいくつかの必要条件を満たしていることになる。これらの条件の依存関係を図示すると図\ref{diagram:VS_cond}のようになる。 

\begin{figure}[htbp]
  \begin{center}
    \begin{tikzpicture} 
      \matrix [column sep=10mm, row sep=5mm] {
        \node (v) [draw, shape=rectangle] {値関数:$\valuefuncB{v}{z}{\xi}{n}$}; \\
        \node (vs) [draw, shape=rectangle] {変分不等式(\ref{eq:variationalInequality})の(一意の)粘性解}; \\
        \node (nessCond) [draw, shape=rectangle] {\shortstack[l]{
            $\bullet\qquad   \delta v -\generator{v}-f=0,  z \in \continuationRegionN\;\qquad\qquad\qquad\qquad\qquad\qquad\cdots$ 式(\ref{eq:DEclassical})\\
            $\bullet\qquad \valuefuncB{v}{z}{\xi}{n} = \max_{\hat{\xi}\in\controlspace(\xi)}\{\valuefuncB{v}{z}{\hat{\xi}}{n-1}-K\},   z \in \stoppingRegionN\;\cdots$  式(\ref{eq:DEclassical})\\
            $\bullet\qquad \vf \in C^1(\real{})\;\qquad\qquad\qquad\qquad\qquad\qquad\qquad\cdots$  式(\ref{formula:v_in_C1})}}; \\
      };
      \draw[<->, double,  thick] (v) -- (vs);
      \draw[->, double,  thick] (vs) -- (nessCond);
    \end{tikzpicture}
  \end{center}
  \caption{粘性解の条件式の依存関係}
  \label{diagram:VS_cond}
\end{figure}

しかしこのような手順により求められた関数は、そのままでは求めるべき値関数$v$であるための十分条件を満たしていない。
つまり一般的には数値計算によりたまたま求められた最初の継続領域上の古典解が、変分不等式(\ref{eq:variationalInequality})の粘性解となっているのかどうかは不明である。
未定係数を決める際の非線型方程式で解が複数ある場合、求めているものがどれなのかは不明である。解析的アプローチなら全ての解を一網打尽に求めることが可能かもしれないが、数値計算の収束計算の場合、一回の処理で求まる解は高々１個だけである。また次の処理でその他の解が求まる保証もないし、そもそもその他の解が存在するかどうかも不明である。存在するかどうか不明なものを探さなくてはならないことになり、数々の困難が立ちはだかる。

本稿の問題の場合、継続領域$\continuationRegionN$を支配する二階線形微分方程式$\delta v -\generator{v}-f=0$の(継続領域上の)古典解はある種の一般解を持つ。
これは二次元の線形独立な基底関数(例えば式(\ref{frm:generalSolution}), $\{\HermiteNu{z}, \HermiteNu{-z}\}$)の線形結合の形式に特殊解を加えた形式の一般解である。
従っていわゆる微分方程式の数値解法を持ち出す必要はない。
一般解が初等関数で表されない場合であっても、
少なくとも関数形が不明なまま行われる数値解析よりは見通し良く求解できる。

\subsection{区分的一般解を用いた粘性解の構築} \label{sec:一般解を用いた粘性解の特定}
ここでは、継続領域ごとに、自由境界問題の条件を満たす一般解を求め、それらを連結することにより粘性解を構築することを考える。
まず、一般に区分的に条件を満たす古典解を単純に連結するだけでは求める粘性解は得られないことを次の節で述べる。

\subsubsection{古典解を接ぎ合わせて粘性解になるのか？} 
本稿の問題の場合、個々の継続領域は自由境界問題という微分方程式系により支配されており、
2階の場合、1個の有限連結継続領域当たり任意定数が2個、未知の自由境界が2個の計4個の未知数を、4本の非線型連立方程式から求めることになる。
しかし一般的に、非線形方程式の実数根は複数あり得る。ただし当然のことながら、境界条件を満たす一般解が複数あっても値関数は一意であるため、値関数に一致する解は一つだけである。
さらに一般にはレジーム毎に、連結した継続領域は$\real{}$上に複数個あり得るため問題は複雑になる。
非線形連立方程式の根の個数が予め1個と特定できればこの問題は回避可能かもしれないが、
一般の方程式の実数根の個数を予め特定することは難しい。これは方程式を構成するパラメータの範囲に依存して数値計算による場合分けが必要な煩雑な問題であり、この手順は求解処理に支障をきたすものである。
境界条件を満たす一般解が複数あるときは、それぞれの解が求める値関数であるかどうか別途検証作業が必要となる。

この問題を緩和するのが粘性解である。
粘性解は古典解の条件を緩和したものであり、同一区間であれば古典解は同時に粘性解である(Lemma \ref{lemma:classical_viscosity}, \ref{prop:solutionOnC}, \ref{theo:classicalPart})。
しかし粘性解の場合、古典解と異なり、定義域に拡張性がある。$C^2$-級でない点も含めて複数の連結継続領域を包含し、実数空間全域を定義域とする粘性解も構築できる。
古典解は独立した継続領域上に個々に存在しているが、本稿の問題では実数全域にわたる古典解は存在しない。つまり継続領域境界では$C^2$-級にならない(Theorem \ref{theo:smoothfitvn})。
実数空間全域を定義域とするような粘性解を構築する場合、個々の連結継続領域上の同一区間に複数存在していた古典解のうち粘性解の条件を満たさないものは排除される。
これにより値関数になり得る古典解候補のみが選別されることになる。

以上の考察により、求める粘性解は、実数全域の定義域のうち、個々の連結継続領域上で古典解となっていることがわかる。
つまり古典解が一般解で表現でき、実数全域をいずれかの連結継続領域に分割することが出来れば、粘性解は、連結継続領域境界以外は古典的一般解で表現できることになる。
式(\ref{eq:relationOSC})-(\ref{eq:seemlessR})により、個々のレジーム$\xi^{(0)}\in \positionspace$毎に、必要なら同時スイッチを全て行って最終継続領域に達した状態別に実数全域を分割して、
\begin{align}  
  \real{}&=\bigoplus_{\begin{subarray}{l} {\xi}^{(p)}\in\controlspace(\xi^{(p-1)})\\ p=1, \cdots, m\\ m=0,1,\cdots \end{subarray}}  \bigl( \stoppingRegionA{n}{\xi^{(0)}\cdots{\xi}^{(m)}} \cap \continuationRegionA{n-m}{{\xi}^{(m)}} \bigr)  \label{eq:R_decompose_to_C}
\end{align}
と書ける。
すなわち実数全域の各点は、境界を除き互いに素な最終継続領域内部にあるといえる。
すなわち、Lemma \ref{theo:classicalPart}より、求める値関数は区分的に分割すれば、何らかの微分方程式の古典的一般解として表現でき、それら古典解を繋ぎ合わせたかたちで書き表わせる。

ただし、単純に区分的な古典的一般解を無条件に繋ぎ合わせるだけでは、実数全域を定義域とする粘性解にはならないことは上述の通りである。
それらの繋ぎ目を適切に処理し、候補となる複数の区分的古典解の中から適切な古典解を選別することによって変分不等式(\ref{eq:variationalInequality})の粘性解は得られることになる。

\subsubsection{区分古典解を用いた粘性解の構築} 
そこで、粘性解を具体的に求めるため、古典的な一般解から粘性解を直接特定できるような手法を考える。
それが可能であれば粘性解の一意性により、求められたものが唯一の粘性解となり、同時にそれが求める値関数に一致することになる。
その手続きを実装するのがTheorem \ref{theo:systemequation_viscosity}である。
現実の具体的問題の多くは個別事情を含んでおり、前提も異なるため汎用的定理が利用できない。そのような場合、定理に相当するものを、その都度自ら構築しなくてはならない。
継続領域$\continuationRegionN$は一般的には以下のように複数の開集合による部分連結領域$\continuationRegionA{n}{\xi\{\beta_i\}}, {\beta_i\in\Xi}$の直和に分解される。
\begin{align} 
  \continuationRegionN=\bigoplus_{\beta_i\in\Xi}\continuationRegionA{n}{\xi\{\beta_i\}} \label{eq:decomp_continuation_region}
\end{align}
個々の連結領域$\continuationRegionA{n}{\xi\{\beta_i\}}\; (\beta_1, \beta_2, \cdots )$に対し、式(\ref{eq:generalsolutionN})の上式で与えられた非同次の一般解のうちの、所定の条件を満たす特殊解を決定することができれば、それが求めるべき$\real{}$全域の変分不等式の粘性解になるというのが、次に示すTheorem \ref{theo:systemequation_viscosity}である。
一般解は式(\ref{eq:generalsolutionN})のような$n$に関する漸化式となっており、再帰的計算処理において、スイッチ権利数$n$の値関数を求める際、スイッチ権利数$n-1$の値関数は既知の関数扱いである。接続先の関数形は既知であるが、どこで接続するかが未知であるような自由境界問題の系であり、滑らかさ条件(regularity condition)を記述する式(\ref{formula:v_in_C1})に基づき、領域が有界の場合、各々の部分連結継続領域$\continuationRegionA{n}{\xi\{\beta_i\}} (\xi\in\positionspace, {\beta_i\in\Xi})$毎に、その各境界両端点で成立する0階、1階の境界条件(smooth pasting condition)による4元連立方程式から、未知数である両自由境界位置、及び2階線形微分方程式の2個の任意定数を確定させる非線型連立方程式を解くという手順になる。

\end{omitlevel1}

In the next theorem the definitions of $\continuationRegion{n}, \stoppingRegion{n}$ defined on equation (\ref{def:stoppingContinuousRegion}) wil be reset and reconstructed from the original definitions
and we use the function $\varphi^{\hat{\xi}}_{n, 1}(z) \equiv \valuefuncB{v}{z}{\hat{\xi}}{n-1} - K\subset\continuousFunction{}(\real{})$.
\begin{theorem}[Piecewise classical viscosity solution] \label{theo:piecewiseClassicalViscositySolution}~\\
  Let Assumption \ref{assumption:structureSwitchingRegion} holds.
  For each $\xi\in\positionspace$, $\valuefuncB{v}{z}{\hat{\xi}}{n-1}$ for ${\hat{\xi}\in\controlspace(\xi)}$ are given value functions (\ref{def:valueFunc_criterion}) and the non-degenerate elliptic PDE: $F(z, V(z), DV(z), D^2V(z))=\delta V -\generator{V}-f(z, \xi)=0$ has a classical particular solution $u_0$ on $\overline{\continuationRegion{n}}$, where
  \begin{align}
    \continuationRegion{n} \equiv \Biggl\{z\in\real{}\Biggm| 
    \begin{split}
      F(z, u_0(z), u_0'(z), u_0''(z))=0\\
      u_0(z) >  \displaystyle\max_{\hat{\xi}\in\controlspace(\xi)}\varphi^{\hat{\xi}}_{n, 1}(z)
    \end{split}  \label{def:C_u}
    \Biggr\}. 
  \end{align}
  Moreover, if 
  \begin{align}
    \stoppingRegion{n}&\equiv\real{}\backslash\continuationRegion{n},\\ 
    u_S(z) &\equiv  \displaystyle\max_{\hat{\xi}\in\controlspace(\xi)}\varphi^{\hat{\xi}}_{n, 1}(z) \text{ on } \stoppingRegion{n}, \label{def:u_S}\\
    u&\equiv 
    \begin{cases} 
      u_0 & \text{ on } \continuationRegion{n},\\
      u_S & \text{ on } \stoppingRegion{n},
    \end{cases} \label{def:u}
    \intertext{ and  $u$  is differentiable on }
    \partial\continuationRegion{n}&\equiv \{z\in\real{}| \forall \varepsilon>0, B_\varepsilon(z)\cap\continuationRegion{n}\ne\phi, B_\varepsilon(z)\cap\stoppingRegion{n}\ne\phi\},  \label{differetiability_u}
  \end{align}
  then
  \begin{align}
    u(z)\in\setViscositySolution{ \min \bigl\{ F(z, V(z), DV(z), D^2V(z)), V  -\max_{\hat{\xi}\in\controlspace(\xi)}\varphi^{\hat{\xi}}_{n, 1}(z) \bigr\}=0} \text{ on } \real{}.  \label{u:solution}
  \end{align}
\end{theorem}

\begin{proof}~\\ 
  We first prove $\continuationRegion{n}$  is an open set.
  From (\ref{differetiability_u}) $u$ is continuous at $z=\bar{z}\in\partial\continuationRegion{n}\cap\stoppingRegion{n}$, i.e.,
  \begin{align}
    u_0(\bar{z})&=u_S(\bar{z})=\displaystyle\max_{\hat{\xi}\in\controlspace(\xi)}\varphi^{\hat{\xi}}_{n, 1}(\bar{z}), \qquad \bar{z}\in\partial\continuationRegion{n}\cap\stoppingRegion{n}.\\
    u_0(z) &>  \displaystyle\max_{\hat{\xi}\in\controlspace(\xi)}\varphi^{\hat{\xi}}_{n, 1}(z) \text{ on } \continuationRegion{n}, 
  \end{align} 
  from (\ref{def:C_u}).
  As both $u_0(z)$ and $\max_{\hat{\xi}\in\controlspace(\xi)}\varphi^{\hat{\xi}}_{n, 1}(z)$ are continuous at $z=\bar{z}$,
  \begin{align}
    \forall z\in\continuationRegion{n}, \exists \varepsilon>0, B_\varepsilon(z)\subset\continuationRegion{n} \Longrightarrow \interior{\continuationRegion{n}}=\continuationRegion{n}. 
  \end{align} 
  Let $G$ be defined as 
  \begin{align}
    G(V)\equiv V-\displaystyle\max_{\hat{\xi}\in\controlspace(\xi)}\varphi^{\hat{\xi}}_{n, 1}(z) \text{ on } \real{}.
  \end{align}

 (i) For $z\in\continuationRegion{n}$,\\
  due to (\ref{def:C_u}),
  \begin{align}
    0&=F(z, u_0(z), u_0'(z), u_0''(z))&<G(u_0(z)),\\
    0&=F(z, u(z), u'(z), u''(z))&<G(u(z)),
  \end{align} 
  which proves $u$ satisfies (\ref{u:solution}).

  (ii) For $\bar{z}\in\partial\continuationRegion{n}$,\\
  from the assumption (\ref{differetiability_u}), $u'(\bar{z})$ exists.

  As $u_0\in\continuousFunction{2}(\overline{\continuationRegion{n}})$, $u_0''(\bar{z})$ exists and
  \begin{align}
    \lim_{z\in\continuationRegion{n}\to\bar{z}}u''(z)=u_0''(\bar{z}).
  \end{align}
  From (\ref{eq:relationOSC})
  \begin{align}
    \exists m\geq 1, 
    \bar{z}\in\partial\continuationRegionN\subset\continuationAfterSwitchingRegionA{n,m}{\xi\xi^{(1)}\cdots{\xi}^{(m)}}&=\stoppingRegionA{n}{\xi\xi^{(1)}\cdots{\xi}^{(m)}} \cap \continuationRegionA{n-m}{{\xi}^{(m)}}. \label{in:bar z in2}
  \end{align}
  From Lemma \ref{prop:test_funct_simultaneous}, 
  \begin{align}
    u_S''(\bar{z})=\{\displaystyle\max_{\hat{\xi}\in\controlspace(\xi)}\varphi^{\hat{\xi}}_{n, 1}\}''(\bar{z})=\varphi^{\xi^{(m)}}_{n, m}{''}(\bar{z}), \text{ as } \varphi^{\xi^{(m)}}_{n, m}\in\continuousFunction{2}(\continuationRegionA{n-m}{{\xi}^{(m)}})  \text{ and } \bar{z}\in\continuationRegionA{n-m}{{\xi}^{(m)}}
  \end{align}
  and
  \begin{align}
    \lim_{z\in\stoppingRegion{n}\to\bar{z}}u''(z)=u_S''(\bar{z}).
  \end{align}
  The semijets of $u$ at $\bar{z}$ are
  \begin{align}
    \begin{cases} 
      \subjets{u(\bar{z})}{B_\varepsilon(\bar{z})}&=\{(u'(\bar{z}), X)| X\leq \min\{u_0''(\bar{z}), u_S''(\bar{z}) \}\},\\
      \superjets{u(\bar{z})}{B_\varepsilon(\bar{z})}&=\{(u'(\bar{z}), X)| X\geq \max\{u_0''(\bar{z}), u_S''(\bar{z}) \}\}.
    \end{cases}
  \end{align}
  Using the above subjet,
  \begin{align}
      \forall (p, X)\in\subjets{u(\bar{z})}{B_\varepsilon(\bar{z})}, &\min\{F(\bar{z}, u(\bar{z}), p, X), G(u(\bar{z}))\}\\
      \geq&\min\{F(\bar{z}, u_0(\bar{z}), u_0'(\bar{z}), \min\{u_0''(\bar{z}), u_S''(\bar{z}) \}), G(u_S(\bar{z}))\}\\
      \geq&\min\{F(\bar{z}, u_0(\bar{z}), u_0'(\bar{z}), u_0''(\bar{z})), G(u_S(\bar{z}))\}\\
      =&\min\{0, 0\}=0.
  \end{align}
  Similarly, using the above superjet,
  \begin{align}
      \forall (p, X)\in\superjets{u(\bar{z})}{B_\varepsilon(\bar{z})}, &\min\{F(\bar{z}, u(\bar{z}), p, X), G(u(\bar{z}))\}\\
      \leq&\min\{F(\bar{z}, u_0(\bar{z}), u_0'(\bar{z}), \max\{u_0''(\bar{z}), u_S''(\bar{z}) \}), G(u_S(\bar{z}))\}\\
      \leq&\min\{F(\bar{z}, u_0(\bar{z}), u_0'(\bar{z}), u_0''(\bar{z})), G(u_S(\bar{z}))\}\\
      =&\min\{0, 0\}=0.
  \end{align}
  These prove $u$ satisfies (\ref{u:solution}).
  
  (iii) For $z\in\stoppingRegionA{n}{\xi\hat{\xi}}\backslash\partial\stoppingRegionA{n}{\xi\hat{\xi}}$ for some $\hat{\xi}\in\controlspace(\xi),$\\
  as $\valuefuncB{v}{z}{\hat{\xi}}{n-1}, {\hat{\xi}\in\controlspace(\xi)}$ is a given value function,
  from Assumption \ref{theo:viscosityvn} and (\ref{eq:VIsuper}), for $n\geq 2$
  \begin{align}
    \valuefuncB{v}{z}{\hat{\xi}}{n-1}&\in\setViscositySolution{ \min \bigl\{ \delta V -\generator{V}-f(z, \hat{\xi}), V  -\max_{\bar{\xi}\in\controlspace(\hat{\xi})}\varphi^{\bar{\xi}}_{n-1, 1}(z)\bigr\}=0} \text{ on } \real{},\\
    &\subset\setViscositySuperSolution{ \delta V -\generator{V}-f(z, \hat{\xi})=0} \text{ on } \real{}. \qquad (\text{from } (\ref{eq:VIsuper}))  \label{eq_continuation:piecewise}
  \end{align}
  Due to Assumption \ref{assumption:structureSwitchingRegion} (ii),
  \begin{align}
    \begin{cases}
      f(z, \hat{\xi})-f(z, \xi)\geq \delta K \text { on } \stoppingRegionA{n}{\xi\hat{\xi}}\subset \stoppingRegionCandidate{\xi\hat{\xi}}{1},\\
      \valuefuncB{v}{z}{\hat{\xi}}{n-1}=\varphi^{\hat{\xi}}_{n, 1}(z)+K \text{ on } \real{},
    \end{cases} \label{eq:cond iii}
  \end{align}
  therefore,
  \begin{align}
    (\ref{eq_continuation:piecewise})\Longrightarrow & \varphi^{\hat{\xi}}_{n, 1}(z)+K &&\in \setViscositySuperSolution{ \delta V -\generator{V}-f(z, \xi)=\delta K} \text { on } \stoppingRegionA{n}{\xi\hat{\xi}}\backslash\partial\stoppingRegionA{n}{\xi\hat{\xi}}\\
    \Longleftrightarrow & \varphi^{\hat{\xi}}_{n, 1}(z) &&\in \setViscositySuperSolution{ \delta V -\generator{V}-f(z, \xi)=0} \text { on } \stoppingRegionA{n}{\xi\hat{\xi}}\backslash\partial\stoppingRegionA{n}{\xi\hat{\xi}} \label{rel:piece_super1}
  \end{align}
  As $\varphi^{\hat{\xi}}_{n, 1}=\max_{\hat{\xi}\in\controlspace(\xi)}\varphi^{\hat{\xi}}_{n, 1}$ on $\stoppingRegionA{n}{\xi\hat{\xi}}$,
  \begin{align}
    \varphi^{\hat{\xi}}_{n, 1} &\in \setViscositySolution{ V -\varphi^{\hat{\xi}}_{n, 1}=0}=\setViscositySolution{ V -\max_{\hat{\xi}\in\controlspace(\xi)}\varphi^{\hat{\xi}}_{n, 1}=0} \text { on } \stoppingRegionA{n}{\xi\hat{\xi}}\label{rel:piece_sol1}
  \end{align}
  Therefore from (\ref{eq:VIsolution}), (\ref{rel:piece_super1}), (\ref{rel:piece_sol1}), (\ref{def:u_S}) and (\ref{def:u}),
  \begin{align}
    u(z)=u_S(z)=\max_{\bar{\xi}\in\controlspace(\xi)}\varphi^{\bar{\xi}}_{n, 1}= \varphi^{\hat{\xi}}_{n, 1}\in\setViscositySolution{ \min \bigl\{ \delta V -\generator{V}-f(z, \hat{\xi}), V  -\max_{\hat{\xi}\in\controlspace(\xi)}\varphi^{\hat{\xi}}_{n, 1}(z)\bigr\}=0} \text{ on } \stoppingRegionA{n}{\xi\hat{\xi}}\backslash\partial\stoppingRegionA{n}{\xi\hat{\xi}}.
  \end{align}
  From (\ref{eq:decompR}) and as all the boundaries between the regions $\stoppingRegionA{n}{\xi\hat{\xi}}, \stoppingRegionA{n}{\xi\bar{\xi}} \text{ and } \continuationRegionA{n}{\xi}, \hat{\xi}\ne\bar{\xi},\; \hat{\xi},\bar{\xi}\in\controlspace(\xi)$ on $\real{}$ are in the case (ii) for each $\xi$ due to (\ref{eq:empty_intersectSS}),
  all the cases (i), (ii) and (iii) prove $u$ satisfies (\ref{u:solution}) on entire $\real{}$.
\end{proof}

According to Corollary \ref{coro:structureSwitchingRegion} for each $\xi\in\positionspace$,
\begin{align}
  \forall z\in\stoppingRegionA{n}{\xi}, \exists\hat{\xi}\in\controlspace(\xi), \alpha\in\Lambda, s.t. 
          z\in{\stoppingRegionA{n}{\xi\hat{\xi}\{\alpha\}}}\subset\stoppingRegionCandidate{\xi\hat{\xi}\{\alpha\}}{1},
\end{align}
where    ${\stoppingRegionA{n}{\xi\hat{\xi}\{\alpha\}}}=\overline{\stoppingRegionA{n}{\xi\hat{\xi}\{\alpha\}}}$ is one connected region if exists,
and each $\stoppingRegionCandidate{\xi\hat{\xi}\{\alpha\}}{1}, \hat{\xi}\in\controlspace(\xi), \alpha\in\Lambda$ is  mutually disjoint.
That is, any switching region ${\stoppingRegionA{n}{\xi\hat{\xi}\{\alpha\}}}$ should belong to the reference region $\stoppingRegionCandidate{\xi\hat{\xi}\{\alpha\}}{1}$ for some $\alpha\in\Lambda$, therefore
we can limit the possible areas of the free boundaries of our free boundary problem within the reference regions.
We can also list all the reference regions disjointedly ordered. 

Therefore, Assumption \ref{assumption:structureSwitchingRegion} ensures the following algorithm to identify each free boundary subproblem. 
\begin{algorithm}[Identifying the system of free boundary problems] \label{algo:system_freeBoundary}~\\
  For each $\xi\in\positionspace$,
  \begin{enumerate}
    \item Sorting all the reference regions $\stoppingRegionCandidate{\xi\hat{\xi}\{\alpha\}}{1}, \hat{\xi}\in\controlspace(\xi), \alpha\in\Lambda$ in order\\
    \item For each pair of adjacent reference regions 
      $\stoppingRegionCandidate{\xi\hat{\xi}\{\alpha_i\}}{1}, \stoppingRegionCandidate{\xi\bar{\xi}\{\alpha_j\}}{1},\; (\stoppingRegionCandidate{\xi\hat{\xi}\{\alpha_i\}}{1}<\stoppingRegionCandidate{\xi\bar{\xi}\{\alpha_j\}}{1}),\; \hat{\xi}, \bar{\xi}\in\controlspace(\xi),\; \alpha_i, \alpha_j\in\Lambda$, solve the following free boundary problem with the smooth pasting condition:
      \begin{itemize}
        \item PDE:
        \begin{align}
          \delta V -\generator{V}-f(z, \xi)	=0
        \end{align}
        on a connected continuation region including the region between $\stoppingRegionCandidate{\xi\hat{\xi}\{\alpha_i\}}{1}$ and $\stoppingRegionCandidate{\xi\bar{\xi}\{\alpha_j\}}{1}$ 
        \item Left free boundary condition:\\
          Left free boundary 
          $\partial\stoppingRegionA{n}{\xi\hat{\xi}\{\alpha_i\}}(+) \in \stoppingRegionCandidate{\xi\hat{\xi}\{\alpha_i\}}{1}$,
          connecting to the function $\varphi^{\hat{\xi}}_{n, 1}(z)$
        \item Right free boundary condition:\\
          Right free boundary 
          $\partial\stoppingRegionA{n}{\xi\bar{\xi}\{\alpha_j\}}(-) \in \stoppingRegionCandidate{\xi\bar{\xi}\{\alpha_j\}}{1}$,
          connecting to the function $\varphi^{\bar{\xi}}_{n, 1}(z)$
      \end{itemize}
      i.e., the continuation region is
      $\continuationRegionA{n}{\xi\{\beta\}}=(\partial\stoppingRegionA{n}{\xi\hat{\xi}\{\alpha_i\}}(+), \partial\stoppingRegionA{n}{\xi\bar{\xi}\{\alpha_j\}}(-))$,  where $\partial\stoppingRegionA{}{}(+), \partial\stoppingRegionA{}{}(-)$ indicate the right and the left  boundaries of the connected region (interval) $\stoppingRegionA{}{}$.
  \item [$\cdot$]
    When the left-most connected reference region is finite, solve another free boundary problem with the infinite continuation region including its immediate left connected region of the left-most connected reference region including $-\infty$ with the asymptotic condition.
  \item [$\cdot$]
    When the right-most connected  reference region is finite, solve another free boundary problem with the infinite continuation region including its immediate right connected region of the right-most connected reference region including $\infty$ with the asymptotic condition.
  \end{enumerate}
\end{algorithm}

\section{Application to an Example in \citeN{suzuki2020optimal}}\label{sec:example_model}
\subsection{Problem Formulation}\label{sec:ProFor}
Suppose that the filtration  $\mathbb{F}=\filtration{\sigmaalgebra}{s}{\TimeSet}$ is generated by a one-dimensional standard Brownian motion $\{W_t: t\in\TimeSet\}$ and there are two stocks, A and B, and suppose that the one-dimensional process $X\equiv\log{\frac{S_A}{S_B}}$, the  spread between the logarithmic price of $S_A$ and that of $S_B$, follows a mean-reverting process (\ref{eq:OU_processs}), called an Ornstein--Uhlenbeck stochastic process. 
The term $dX$ is considered to be the infinitesimal spread return between the pair of stocks. For $t\in\TimeSet$, the dynamics of the infinitesimal spread return is formulated as follows:
\begin{align} 
  & \ousde{X}{t}{\theta}{\mu}{\sigma}{W}.  \label{eq:OU_processs} 
\end{align}
The parameter $\mu$ is a constant representing the long-run mean level of the process $X$, and $\theta > 0$ is the intensity of the mean-reverting process. The parameter $\sigma>0$ is the volatility of $X$.
Using the following change of variable from $X$ to $Z$,
\begin{align} 
  X-\mu=\dfrac{\sigma}{\sqrt{\theta}}Z \label{eq:XZ_transform},
\end{align}
equation (\ref{eq:OU_processs}) is transformed into 
\begin{align} 
  & \ousdeA{Z_t}{\theta}{t}{W_t},   \label{eq:OUprocessZ}
\end{align}
which is solved by the following formula for $s\geq t$:
\begin{align} 
  Z_s&=Z_t e^{-\theta(s-t)} + \sqrt{\theta} \intA{t}{s}{\discount{-\theta(s-u)}}{W_u}.   \label{eq:Solution-Z}  
\end{align}
The corresponding infinitesimal generator for $Z$ is $\generatorx{}{} \equiv \dfrac{\theta}{2}\dfrac{d^2}{d z^2}-\theta z \dfrac{d}{dz}$.
The set of regimes (switching states) $\positionspace$ is composed of three different positions $\{-1, 0, 1\}$ for which the investor can choose, where
\begin{align} 
  \positionspace : 
  \begin{cases}
    1 & \cdots (\text{long A/ short B}), \\
    -1 & \cdots (\text{long B/ short A}), \\
    0 & \cdots (\text{square}).
  \end{cases} \label{eq:positionspace}
\end{align} 
$\controlspace{(\xi)}$ denotes the feasible set of the states (positions) into which the current state $\xi$ can transfer; i.e., 
\begin{align} 
    \controlspace{(\xi)}\equiv \{\hat{\xi} | \hat{\xi}  \in \positionspace,   \abs{\hat{\xi} -\xi}=1 \}, \quad \xi\in\positionspace, \label{eq:control_space0}
\end{align}
which is reformulated as:
\begin{align} 
  \controlspace{(\xi)}=
  \begin{cases}
    \{0\} & \cdots (\xi\in\{\pm1\}),\\
    \{\pm1\} & \cdots (\xi=0).
  \end{cases} \label{eq:control_space}
\end{align}
Our objective function is composed of the gain (cumulative spread return) from the spread position less the transaction costs associated with the change in the positions.
The objective function $J$ is formulated as follows.
\begin{align} 
  \begin{split}
     \valuefuncC{J}{z}{\xi}{n}{\alpha} \equiv &  \ExpectationCond{\intA{{t_0}}{\infty}{\discount{-\delta(s-{t_0})}\process{\xi}{s}}{\process{X}{s}} 
       -\lambda\intA{{t_0}}{\infty}{\discount{-\delta(s-{t_0})}\{\process{\xi}{s}(\process{X}{s}-\mu)\}^2}{s} \\ 
     &  - \sum_{\tau_i\geq {t_0}}^{\tau_n}\discount{-\delta(\tau_i-{t_0})}  \abs{\xi_{\tau_{i}} - \xi_{\tau_{i-1}}} K}{Z_{t_0}^-=z, \xi_{t_0}^-=\xi},  \label{eq:performanceCriteria0} 
  \end{split}
\end{align} 
where $\delta>0$ is the time-discount rate. For convenience, we set $\tau_{n+1}\equiv \infty$.
We define the quadratic running reward function as: 
\begin{align} 
  f(z, \xi)\equiv -(\xi\sqrt{\theta}\sigma z + \lambda\abs{\xi}\dfrac{\sigma^2}{\theta}z^2). \label{def:f} 
\end{align}

\begin{omitlevel1}
\begin{lemma}[Uniformly integrable martingale] \label{theo:martingalePart}
  For any stopping time $\tau (\tau\geq t)$, the following equation holds.
\begin{align} 
  \expectationx{x}{\int_{t}^{\tau}{\discount{-\delta(s-t)}}{\{\process{\xi}{s}d\process{X}{s}-\lambda\{\process{\xi}{s}(\process{X}{s}-\mu)\}^2ds\}}} = \expectationx{z}{\intA{t}{\tau}{\discount{-\delta(s-t)}f(Z_s, \xi)}{s} } \label{eq:trans_timeIntegral} 
\end{align}
\end{lemma} 

Using Lemma \ref{theo:martingalePart} with (\ref{eq:control_space0}) and (\ref{def:controlspaceA}), the performance criterion (\ref{eq:performanceCriteria0}) is reformulated as follows (note that $\tau_{n+1}=\infty$):
\begin{align} 
  \begin{split} 
    \valuefuncC{J}{z}{\xi}{n}{\alpha}  &=  \expectationCond{\sum_{i=0}^n\intA{\tau_i}{\tau_{i+1}}{\discount{-\delta(s-t)}{f(Z_s, \process{\xi}{\tau_i})}}{s} 
       - \sum_{\tau_i\geq t_0}^{\tau_n}\discount{-\delta(\tau_i-t_0)}K}{Z_{t_0}^-=z, \xi_{t_0}^-=\xi}. \label{criteria} 
  \end{split} 
\end{align}

From (\ref{eq:Solution-Z}), for $s\geq t$, $\expectationCond{Z_s}{t}=Z_t e^{-\theta(s-t)}$.
\begin{align} 
  \text{From } Z^2_s&=Z_t^2 e^{-2\theta(s-t)} + 2\sqrt{\theta}Z_te^{-\theta(s-t)} \intA{t}{s}{\discount{-\theta(s-u)}}{W_u}+  \theta\biggl(\intA{t}{s}{\discount{-\theta(s-u)}}{W_u}\biggr)^2, \\ 
  \expectationCond{Z^2_s}{t}&=Z_t^2 e^{-2\theta(s-t)} + \theta\intA{t}{s}{\discount{-2\theta(s-u)}}{u} =(Z_t^2-1/2) e^{-2\theta(s-t)} + 1/2.  \label{eq:Solution-Z2}  
\end{align}
 
\end{omitlevel1}

Buy-and-hold strategy is a special case of our problems without any options to switch corresponding to (\ref{eq:variationalInequality0}), i.e., $n=0$.
It is important since it represents an initial condition for the recurrence formula in 
the following family of iterative switching problems. As the feasible control space is empty ($\controlspaceAn{0}(\xi)=\phi$),
the value function (\ref{def:valueFunc_criterion}) for each $\xi\in\positionspace$ and $n=0$ is calculated as follows: 
\begin{align} 
  \begin{split} 
    \valuefuncA{\widehat{v}}{z}{\xi} & \equiv  \valuefuncB{v}{z}{\xi}{0} = \valuefuncC{J}{z}{\xi}{0}{\phi}
     = \intA{{t_0}}{\infty}{\discount{-\delta(s-{t_0})}\expectationx{z}{f(Z_s, \xi)}}{s} = -(k_2\abs{\xi}z^2+k_1 \xi z+k_0\abs{\xi}), \\
    & \text{where} \quad k_2\equiv\dfrac{\lambda\sigma^2}{\theta(2\theta+\delta)}=\dfrac{\delta}{\theta}k_0>0,  k_1\equiv \dfrac{\sqrt\theta\sigma}{\delta+\theta}>0, k_0\equiv\dfrac{\lambda\sigma^2}{\delta(2\theta+\delta)}>0. \label{eq:v0}
  \end{split} 
\end{align}
According to \citeN{suzuki2020optimal} the above model satisfies all the assumptions provided in Section \ref{sec:assumptions_problem}.

\begin{theorem}[Dynamic programming principle: \citeN{suzuki2016optimal}] \label{theo:DPP}~\\
  For any  $z\in\real{}, n\geq 0$, $m (0\leq m\leq n)$ and any stopping time $\tau$, satisfying $t_0=\tau_0\leq\tau_1\leq\cdots\leq\tau_m\leq\tau < \tau_{m+1}\leq\tau_n$, where $\{\tau_i\}$ is a series of optimal switching time. Then the following equation holds. 
\begin{align} 
  \begin{split} 
    \valuefuncB{v}{z}{\xi}{n}  
     =& \sup_{\alpha \in \controlspaceAn{n}(\xi)}  \ExpectationCond{\intA{{t_0}}{\tau}{\discount{-\delta(s-{t_0})}{f(Z_s, \process{\xi}{s})}}{s}  - \sum_{{t_0}\leq\tau_i\leq\tau}^{}\discount{-\delta(\tau_i-{t_0})}K \\
       & + \discount{-\delta(\tau-{t_0})}\valuefuncB{v}{Z_\tau}{\xi_\tau}{n-m}}{Z_{t_0}^-=z, \xi_{t_0}^-=\xi}. \label{eq:DPP} 
  \end{split} 
\end{align}
\end{theorem} 

\begin{omitlevel1}
From Theorem \ref{theo:DPP}, for all  $z\in \real{}, \xi\in \positionspace \text{ and } n (1\leq n\leq +\infty)$, following inequalities hold and one of them holds as equality.
\begin{align}
  &\begin{cases}
     \intertext{(i). Considering the case $m=0$, for general stopping time $\tau>t_0$, the following inequality holds: }
     \begin{split}
       \valuefuncB{v}{z}{\xi}{n}  \geq 
       &\expectationCond{\intA{{t_0}}{\tau}{\discount{-\delta(s-{t_0})}{f(Z_s, \process{\xi}{s})}}{s} 
         + \discount{-\delta(\tau -{t_0})}\valuefuncB{v}{Z_{\tau}}{\xi_\tau}{n} 
       }{Z_{t_0}=z, \xi_{t_0}=\xi}. 
     \end{split} 
     \intertext{(ii). Considering the case $m=1$ and $\tau=t_0=\underbarA{\tau}_1$, where $\underbarA{\tau}_1$ is a general switching time (not necessarily optimal),   the following inequality holds:}
     \begin{split}
       \valuefuncB{v}{z}{\xi}{n}  \geq  & \max_{\hat{\xi}\in\controlspace(\xi)}\{\valuefuncB{v}{z}{\hat{\xi}}{n-1}-K\}.
     \end{split}  
   \end{cases}  \label{eq:DynamicProgrammingN}
\end{align}
Moreover, the case (ii) is divided into the following two cases (A) and (B), depending on the optimal switching time $\tau_1$.
\\\noindent(A). $\underbarA{\tau}_1$ is the optimal switching time. ($t_0=\tau=\underbarA{\tau}_1=\tau_1$): 
From (\ref{eq:DPP}), 
\begin{align}
  \valuefuncB{v}{z}{\xi}{n}  = & \max_{\hat{\xi}\in\controlspace(\xi)}\{\valuefuncB{v}{z}{\hat{\xi}}{n-1}-K\}. \label{eq:switchingEq}
\end{align}
Conversely if equation (\ref{eq:switchingEq}) holds, there exists $\hat{\xi}\in\controlspace(\xi)$, such that 
$\valuefuncB{v}{z}{\xi}{n}  = \valuefuncB{v}{z}{\hat{\xi}}{n-1}-K$. This means the immediate switching from $\xi$ to $\hat{\xi}$ is optimal. Therefore, $t_0=\underbarA{\tau}_1=\tau_1$.
\\\noindent(B). $\underbarA{\tau}_1$ is not the optimal switching time. ($t_0=\tau=\underbarA{\tau}_1<\tau_1$): \\
From (A), in this case strict inequality:
$\valuefuncB{v}{z}{\xi}{n}  >  \max_{\hat{\xi}\in\controlspace(\xi)}\{\valuefuncB{v}{z}{\hat{\xi}}{n-1}-K\}$ holds.
%

Suppose that the current position is $\xi\in\positionspace$ at time $t$  with $n$ options to switch, then $Z_t\in\stoppingRegion{n}$ means the process $Z_t$ is on the optimal switching region and $\xi$ should be switched immediately to another position. 
If the optimal switching is  from $\xi$ to $\hat{\xi}$, we write $Z_{t}\in\stoppingRegionA{n}{\xi\hat{\xi}}$. 
$Z_{t}\in\continuationRegion{n} $ means the process $Z_{t}$ is on the continuation region and the optimal strategy suggests that the position $\xi$ should be kept at $t$.
Therefore, when $\xi_{\tau_i}=\xi\; (i\geq 0)$, the optimal switching time $\tau_{i+1} $ after $\tau_i$ is formally expressed as:
\begin{align}
  \tau_{i+1}&=\inf\{s\geq \tau_i|Z_s\in\stoppingRegion{n-i}\} \; (i\geq 0). \label{eq:tau}
\end{align}  
The value function of the switching problem is constructed by the recurrence formula as in the bottom formula of (\ref{def:stoppingContinuousRegion}). Suppose that you have $n$ options to switch positions, then after switching once optimally, the resulting problem is the same problem with $n-1$ options to switch.
\end{omitlevel1}

\begin{omitlevel3}
\begin{itemize}
\item[(i)] Non-empty condition:\\
  According to Lemma \ref{theo:nonempty}, for $\hat{\xi}\in\controlspace(\xi), \xi\in\positionspace$,
  \begin{align}
    \hat{v}(z, \xi)\leq\hat{v}(z, \hat{\xi})-K \Longleftrightarrow
    k_2(\abs{\hat{\xi}}-\abs{\xi})z^2+k_1({\hat{\xi}}-{\xi})z+\{k_0(\abs{\hat{\xi}}-\abs{\xi})  +K\}\leq 0. \label{ineq:non-empty2}
  \end{align}
  \begin{itemize}
  \item If $({\xi}, \hat{\xi})=(\pm1, 0)$, \\
    (\ref{ineq:non-empty2}) is  
    \begin{align}
      -k_2z^2\mp k_1z+(K-k_0)  \leq 0,
    \end{align}
    which has a solution $z\in\real{}$ as $k_2>0$.
    Therefore $\stoppingRegionA{n}{\pm1}\ne\phi, \forall n\geq 1$ from Lemma \ref{theo:nonempty}.
    
  \item If $({\xi}, \hat{\xi})=(0, \pm1)$,\\
    (\ref{ineq:non-empty2}) is 
    \begin{align}
      k_2z^2\pm k_1z+(K+k_0)  \leq 0, \label{ineq:non-emptyii}
    \end{align}
    whose determinant is $D\equiv k_1^2-4\delta k_0(K+k_0)/\theta$.
    If $D<0$ then the left hand side of the inequality (\ref{ineq:non-emptyii}) does not have the points of intersection with the horizontal $z$-axis. Therefore we assume the case:
    \begin{align}
      D\geq 0\Longleftrightarrow 4\lambda(K+k_0)\leq \theta^2(2\theta+\delta)/(\delta+\theta)^2. \label{cond:D}
    \end{align}
    In this case $\stoppingRegionA{n}{0}\ne\phi, \forall n \geq 1$. This condition will be examined for particular values of parameters in (\ref{rel:lambda_b_D}).
  \end{itemize}
  
\item[(ii)] Reference regions:\\
  We identify the set $\stoppingRegionCandidate{\xi\hat{\xi}}{1}$.
  \begin{align}
    z\in\stoppingRegionCandidate{\xi\hat{\xi}}{1} &\Longleftrightarrow 
    f(z, \hat{\xi})-f(z, \xi)\geq \delta K \Longleftrightarrow (\abs{\xi}-\abs{\hat{\xi}})\frac{\lambda\sigma^2}{\theta}z^2+({\xi}-{\hat{\xi}})\sqrt{\theta}\sigma z-\delta K\geq 0  \label{ineq:refRegion}
  \end{align}
  \begin{itemize}
  \item If $({\xi}, \hat{\xi})=(\pm1, 0)$,\\
    then (\ref{ineq:refRegion}) is equivalent to:
    \begin{align}
      (z-p(\mp1+\sqrt{1+q}))(z-p(\mp1-\sqrt{1+q}))\geq 0,
    \end{align}
    where $q:=4\delta\lambda K/\theta^2>0$ and $p:=\theta^{(3/2)}/(2\lambda\sigma)$. Therefore 
    \begin{align}
      \stoppingRegionCandidate{\pm1 0}{1}=\big(-\infty, p(\mp1-\sqrt{1+q})\big] \cup \big[p(\mp1+\sqrt{1+q}), +\infty\big). \label{regionQ1} 
    \end{align}
    
  \item If $({\xi}, \hat{\xi})=(0, \pm1)$ and $q\leq 1$,\\
    then (\ref{ineq:refRegion}) is equivalent to:
    \begin{align}
      (z-p(\mp1+\sqrt{1-q}))(z-p(\mp1-\sqrt{1-q}))\geq 0.
    \end{align} 
    In this case, if $q>1$, then from Lemma \ref{prop:candidateRegion} $\stoppingRegionCandidate{0 \pm1}{1}=\phi\Rightarrow \continuationAfterSwitchingRegionA{n,1}{0\:\pm1}=\phi$,
    また$\xi=0$からの連続同時スイッチはないため$\continuationAfterSwitchingRegionA{n,1}{0\:\pm1}=\stoppingRegionA{n}{0}=\phi$, $\continuationRegionA{n}{0}=\real{}\; (n\geq 0)$. つまりコスト$K$が巨大で$q>1$の場合、一旦$\xi=0$に遷移すると、$n$に関わらず再度$\xi=\pm1$へとポジション構築されなくなる。
    therefore we assume otherwise, i.e.,
    $q\leq 1$のとき、パラメータの前提としては、
    
    \begin{align}
      \lambda K\leq\theta^2/(4\delta)\Longleftrightarrow q\leq 1. \label{cond:q}
    \end{align}
    であるが、このとき式(\ref{ineq:refRegion})は次式と同値になる(複号同順)。
    \begin{align}
      (z-p(\mp1+\sqrt{1-q}))(z-p(\mp1-\sqrt{1-q}))\geq 0\nonumber
    \end{align}
    In this case, 
    \begin{align}
      \stoppingRegionCandidate{0 \pm1}{1}=\big[p(\mp1-\sqrt{1-q}), p(\mp1+\sqrt{1-q})\big] \; \cdots\; \text{(double sign in the same order)}.\label{regionQ2} 
    \end{align}
  \end{itemize}
    
\item[(iii)] Assumption \ref{assumption:structureSwitchingRegion}成立の検証:\\
  
  We continue to assume $q\leq 1$.
  From $\sqrt{1+q}-1<1-\sqrt{1-q}, q\in(0, 1]$, the relations:
  \begin{align}
    \stoppingRegionCandidate{1 0}{1}\supset\stoppingRegionCandidate{0 -1}{1} \text{ and } \stoppingRegionCandidate{-1 0}{1}\supset\stoppingRegionCandidate{0 1}{1} \label{relation:Q}
  \end{align}
  hold (See Figure \ref{chart:example_quadratic}).
  For $m\geq 1$, all the permutations of the full $m$-simultaneous switchings in this example are listed as follows.
  \begin{align}
    m=1:& (\xi^{(0)}, \xi^{(1)})\in \{(0, \pm1), (\pm1, 0)\}\\
    m=2:& (\xi^{(0)}, \xi^{(1)}, \xi^{(2)})\in \{(1, 0,-1), (-1, 0, 1)\}
  \end{align}
  Therefore from Proposition \ref{prop:sufficientCond} with (\ref{relation:Q}), Assumption \ref{assumption:structureSwitchingRegion} (ii), i.e.,
  \begin{align}
    \stoppingRegionCandidate{\xi\hat{\xi}}{1}\supset \stoppingRegionA{n}{\xi\hat{\xi}},\quad \forall \hat{\xi}\in\controlspace(\xi), \forall\xi\in\positionspace
  \end{align}
  is satisfied.
  On the other hand, when $\xi=0$, $\stoppingRegionCandidate{0-1}{1}\cap\stoppingRegionCandidate{0 1}{1}=\phi$. When $\xi=\pm1$, $\{\controlspace(\xi)\}$ is singleton respectively,
  therefore Assumption \ref{assumption:structureSwitchingRegion} (i) is also satisfied. See Figure \ref{chart:example_quadratic} for more understanding.
  
\item[(iv)] 連結スイッチ領域の特定:\\
  Next we consider Corollary \ref{coro:structureSwitchingRegion}.
  From  (\ref{regionQ2}), $\stoppingRegionCandidate{0\pm1}{1}$ is a single region respectively, and from (\ref{regionQ1}) each $\stoppingRegionCandidate{\pm1 0}{1}$ consists of two regions, respectively. 
  Therefore we define:
  \begin{align}
    \stoppingRegionCandidate{\pm1 0\{-\}}{1}\equiv\big(-\infty, p(\mp1-\sqrt{1+q})\big], \label{def:regionQ10-}\\
\stoppingRegionCandidate{\pm1 0\{+\}}{1}\equiv\big[p(\mp1+\sqrt{1+q}), +\infty\big). \label{def:regionQ10+}
  \end{align}
  ここで$\stoppingRegionCandidate{\pm1 0}{1}=\stoppingRegionCandidate{\pm1 0\{+\}}{1} \oplus \stoppingRegionCandidate{\pm1 0\{-\}}{1}$ (複号同順).
  Assumption \ref{assumption:structureSwitchingRegion}の成立により
  Then from Corollary \ref{coro:structureSwitchingRegion}, each set of 
  \begin{align}
    &\stoppin1gRegionA{n}{0\pm1}\subset  \stoppingRegionCandidate{0\pm1}{1} \label{in:S01}\\
    &\stoppingRegionA{n}{\pm1 0\{-\}}\subset  \stoppingRegionCandidate{\pm1 0\{-\}}{1}\label{in:S10-}\\
    &\stoppingRegionA{n}{\pm1 0\{+\}}\subset  \stoppingRegionCandidate{\pm1 0\{+\}}{1}\label{in:S10+}
  \end{align}
  is a single connected region respectively, i.e., there are 6 (connected) regions over all $\xi\in\{0, \pm1\}$.
  The structure of the switching regions of the above considered example is illustrated in Figure \ref{chart:example_quadratic}.
\end{itemize} 
  
Note that similar to the symmetric property of the value function in Lemma \ref{theo:symmetryVf}, the optimal strategies are also proved to be symmetric.
The reference regions $\stoppingRegionCandidate{0 \pm1}{1}$ and $\stoppingRegionCandidate{\pm1 0}{1}$ in (\ref{regionQ2}), (\ref{def:regionQ10-}) and (\ref{def:regionQ10+}) also have the symmetric properties and  the corresponding optimal switching regions also have.
\end{omitlevel3}
  
\begin{omitlevel1}
\begin{figure}[htbp]
  \begin{center}
    \includegraphics[width=12cm,bb=80 480 520 740,clip,angle=0]{quadratic_regions.eps}    
    \caption{The structure of the reference regions in Example \ref{ex:identifyingSwitchingRegions2}}
    \label{chart:example_quadratic}
  \end{center}
  \begin{flushleft}
    Each of the  6 reference regions $\stoppingRegionCandidate{0 \pm1}{1}$ and $\stoppingRegionCandidate{\pm1 0\{\pm\}}{1}$ in (\ref{in:S01}), (\ref{in:S10-}) and (\ref{in:S10+}) includes at most only one connected switching region of $\stoppingRegionA{n}{0\pm1}$, $\stoppingRegionA{n}{\pm1 0\{\pm\}}$, respectively in it.
    From Lemma \ref{theo:symmetryVf},    all the regions have the symmetric properties.
    The graphs show the case under the conditions of (\ref{cond:D}) and (\ref{cond:q}).
    This is the new result of the application of our main theorems in this paper to our particular problem with the quadratic running function in the objective function.
  \end{flushleft}
\end{figure}
\end{omitlevel1}

\begin{omitlevel2}
証明の過程で以下のような注意が必要である。
\begin{remark}[有限・無限連結領域に関わる未知数と方程式数] \label{rem:unknown_parameters}~\\
  其々の連結部分継続領域に関して、領域が有界の場合、その各境界両端点で成立する0階、1階の計4個の境界条件が成立し、その4元連立方程式から、両自由境界位置及び2階線形微分方程式の2個の任意定数、合計4個の未知数を特定するような非線型連立方程式を解くことになる。

  一方、領域が有界でない場合には、有界でない側の領域の端点は$\pm\infty$という既知数になり、また式(\ref{lim:infty})から任意定数1つが確定するため、未知数が2個減る一方で、有界でない側の端点で上記境界条件が成立しない分、条件式の本数も2本減ることになり、方程式の数と未知数の数の個数の関係は整合的である。
\end{remark}

\begin{remark}[自由境界による区分継続領域上の古典解の一意性] \label{rem:freeboundary_classical_system_uniqueness}~\\
Theorem \ref{theo:systemequation_viscosity}により、各区分連結継続領域上の微分方程式の古典的一般解から直接粘性解を求める手順が明らかになった。
これにより、図 \ref{diagram:VS_cond}で問題となっていた、必要条件のみから粘性解を特定する際に生じた複数根から真の解を特定する煩雑さから解放されることになった。
すなわち、手順通りに行った結果、最初に発見された、条件を満たす根から構成された解が唯一の解となり、それがそのまま求める値関数となり、その時一緒に求められた継続領域、スイッチ領域が最適戦略を表現していることになる。
他の解の可能性に悩む必要はもはやなくなった。
\end{remark}
\end{omitlevel2}

\subsection{Calculation of the Piecewise Classical Solution} 

\def\CC#1#2{C^{#1, #2}_n}
\def\C#1{C^{#1}_n}


\subsubsection{The Classical General Solution on the Continuation Region} 

\begin{omitlevel2}
求める値関数$\vf$は変分不等式(\ref{eq:variationalInequality})の粘性解であったが、その変分不等式に含まれる微分方程式(\ref{eq:PDE_callasical_0})の一般解をここで求める。
その微分方程式は次のような非斉次微分方程式として表わされる。
\begin{align}
  &\DEOUnonhomo{V}{z}{\xi}=0 \quad \cdots (\xi\in\positionspace) \label{eq:DEOUnonhomo1}
\end{align}
この非斉次微分方程式に対応する斉次微分方程式はHermite 微分方程式である。
Hermite微分方程式の解として自由度$\nu$のHermite関数$\HermiteNu{z}$が知られている(\citeN{Lebedev:72}, \citeN{Seaborn:91}参照。)。

この関数$\HermiteNu{z}$の積分表示形式は、
\begin{align} 
  \begin{split} 
    \HermiteNu{z}&=\HermiteNuIntegralRepresentation{z}\\
    &=\HermiteNuIntegralRepresentationA{z} >0,  (z\in \real{}, \nu<0).
  \end{split}  \label{def:Hermite_integral}
\end{align} 
ただし二階線形微分方程式の一般解を構成するためには、基底解すなわち線形独立な2個の基本解の組を特定しなくてはならない。
式(\ref{def:Hermite_integral})で与えられるエルミート関数は第一種エルミート関数と呼ばれているが、これと対を成す線形独立な関数としては、一般的には第二種エルミート関数が用いられるのが普通である
\footnote{第一種エルミート関数は合流型超幾何関数を用いて式(\ref{eq:def_Hermite})のように表されるが、第二種エルミート関数も同様に合流型超幾何関数を用いて表現される。}。 
つまり\{第一種エルミート関数, 第二種エルミート関数\}という組による基底解である。この基底解を用いる場合、あらゆる$\nu$の値に対して両基本解は互いに線形独立となる。
しかし、本稿ではこの基底解を用いず、代わりに、より単純な基底解を用いる。
微分方程式(\ref{eq:DEOUnonhomo1})の独立変数$z$の代わりに$-z$を独立変数とするような微分方程式が同一の微分方程式になることから、$\HermiteNu{-z}$も同一の微分方程式の解になる。
すなわち、第一種エルミート関数のみを用いて、
\begin{align}
  &\{\HermiteNu{z}, \HermiteNu{-z}\} \label{frm:generalSolution}
\end{align}
を基底解とする。本稿ではこちらを用いて一般解を形成している。
この場合、両者の線形独立性確認のためのロンスキー行列は、
\begin{align}
  &\Wronskian{\HermiteNu{z}, \HermiteNu{-z}}=\dfrac{2^{\nu+1}\sqrt{\pi}}{\Gamma(-\nu)}e^{z^2} \neq 0, (\nu \neq 0,1,\cdots)\nonumber
\end{align}
であり、$\nu$が非負整数でなければ式(\ref{frm:generalSolution})が2階線型微分法定式(\ref{eq:DEOUnonhomo1})の斉次方程式の線形独立な基本解を形成しているといえる。
本稿の問題では常に$\nu\equiv -\dfrac{\delta}{\theta}<0$であるため線形独立性は担保されている。
従って、本稿はこの基底(\ref{frm:generalSolution})を、微分方程式(\ref{eq:DEOUnonhomo1})の斉次方程式の一般解として用いる。
一般解の基底関数(\ref{frm:generalSolution})は単調増加関数と減少関数の組合せになっている。

\end{omitlevel2}

The classical general solution on the continuation region to the PDE:
\begin{align}
  \delta V -\generator{V}-f(z, \xi)=0 \Longleftrightarrow 
  \DEOUnonhomo{V}{z}{\xi}=0 \quad \cdots (\xi\in\positionspace) \label{eq:DEOUnonhomo2}
\end{align}
is formulated as (\ref{eq:nonhomo_general_sol}).
Since $\widehat{v}$ is a particular solution to the differential equation (\ref{eq:DEOUnonhomo2}), the general solution of this non-homogeneous equation is formulated as follows: for $n\geq 0$ and $\xi\in\positionspace$,
\begin{align}
  &\vf=\CC{\xi}{+}\HermiteNu{z}+\CC{\xi}{-}\HermiteNu{-z}+  \vfhat, \cdots  z\in\continuationRegion{n}, \label{eq:nonhomo_general_sol}
\end{align}
where $\{\CC{\xi}{+}, \CC{\xi}{-} \}, \; \xi\in\positionspace$ are the arbitrary constants.
Substituting the above general solution to the PDE for the formula (\ref{eq:DEclassical}), the value function is fomulated as follows.
For each $n \geq 1, \xi\in\positionspace$,
\begin{align}
  \valuefuncB{v}{z}{\xi}{n}=
  &\begin{cases}
    \CC{\xi}{+}\HermiteNu{z}+\CC{\xi}{-}\HermiteNu{-z}+  \vfhat, & z\in\continuationRegion{n}, \\
    \displaystyle\max_{\hat{\xi}\in\controlspace(\xi)}\{\valuefuncB{v}{z}{\hat{\xi}}{n-1}-K\}, & z\in \stoppingRegion{n}.
  \end{cases} \label{eq:generalsolutionN}
\end{align}

In order to identify free boundaries, see \citeN{suzuki2020optimal} example 4.2.

\begin{omitlevel2}
ここでHermite 関数の基本的性質を示す。
$\nu<0$ のときHermite 関数は以下の性質を持つ。
\begin{align} 
  &\HermiteNu{z} > 0 \quad (z\in\real{}) \qquad \nonumber\\
  &\HermiteNu{z} \to 0 \quad (z\to\infty) \qquad  \nonumber\\
  &\HermiteNu{z} \to \infty \quad  (z\to-\infty)   \label{lim:H-} \\
  &\HermiteNu{z}'= 2\nu\Hermite{\nu-1}{z}<0 \cdots  \cdots \text{(単調減少性)}  \nonumber\\
  &\HermiteNu{z}''= 4\nu(\nu-1)\Hermite{\nu-2}{z}>0  \cdots \cdots \text{(凸性)}  \nonumber\\
  &\HermiteNu{z} \in \continuousFunction{\infty}(\real{})  \cdots  \cdots \text{(無限回可微分性)}  \nonumber\\
  &\HermiteNu{0} =\dfrac{2^\nu\Gamma(\frac{1}{2})}{\Gamma(\frac{1-\nu}{2})}  \nonumber\\ 
  &\HermiteNu{z}-\HermiteNu{-z}<0 \; (z>0) \label{ineq:Hermite} 
\end{align} 
数値計算用に以下の表現が役立つ。
\begin{align} 
 \HermiteNu{z}&=\dfrac{1}{2\Gamma(-\nu)}\sum_{m=0}^\infty\dfrac{(-1)^m\Gamma(\frac{m-\nu}{2})}{m!}(2z)^m, \abs{z}<\infty  \label{eq:HermiteCalc}
\end{align} 
詳細は\citeN{Lebedev:72}等を参照。
\end{omitlevel2}

\subsubsection{Application of Algorithm \ref{algo:system_freeBoundary} to the Example}\label{sec:Application of Algorithm}

We will apply Theorem \ref{theo:piecewiseClassicalViscositySolution} with Algorithm \ref{algo:system_freeBoundary}  to the example.
According to Algorithm \ref{algo:system_freeBoundary},  for each regime $\xi\in\positionspace$, for each pair of adjacent reference regions 
$(\stoppingRegionCandidate{\xi\hat{\xi}\{\alpha_i\}}{1}, \stoppingRegionCandidate{\xi\bar{\xi}\{\alpha_j\}}{1})$ and their containing continuation region  $\continuationRegionA{n}{\xi\{\beta\}}, \beta\in\Xi$,

Within the domain of the function $v$ in (\ref{eq:generalsolutionN}), the classical non-homogeneous general solution on each part of the connected continuation region $\continuationRegionA{n}{\xi\{\beta\}}\subset\continuationRegionA{n}{\xi} (\xi\in\positionspace, {\beta\in\Xi})$ is formulated as follows.
For each $n, \xi\in\positionspace, \beta\in\Xi$,
\begin{align}
  u_n^{\xi\{\beta\}}(z)&\equiv\CC{\xi}{+\{\beta\}}\HermiteNu{z}+\CC{\xi}{-\{\beta\}}\HermiteNu{-z}+  \vfhat, & z\in\continuationRegionA{n}{\xi\{\beta\}}. \label{eq:general_u0}
\end{align}
If we denote each boundary of the connected continuation region $\continuationRegionA{n}{\xi\{\beta\}}$ as
\begin{align}
  z^{\partial\{-\}}\equiv\partial\stoppingRegionA{n}{\xi\hat{\xi}\{\alpha_i\}}(+), \\
  z^{\partial\{+\}}\equiv\partial\stoppingRegionA{n}{\xi\bar{\xi}\{\alpha_j\}}(-), 
\end{align}
where $z^{\partial\{-\}}< z^{\partial\{+\}}$.
If the both boundaries are finite, from (\ref{in:bar z in}),
\begin{align}
  \exists\, m_1, m_2\geq 1,
  \begin{cases}
    z^{\partial\{-\}}&\in\partial\continuationRegionA{n}{\xi\{\beta\}}\cap\continuationRegionA{n-m_1}{\hat{\xi}^{(m_1)}\{-\}},\\
    z^{\partial\{+\}}&\in\partial\continuationRegionA{n}{\xi\{\beta\}}\cap\continuationRegionA{n-m_2}{\bar{\xi}^{(m_2)}\{+\}},
  \end{cases}
  \qquad \cdots \quad \hat{\xi}, \bar{\xi}\in\controlspace(\xi), \xi\in\positionspace, {\beta\in\Xi}, \nonumber
\end{align}
Using the bottom equation of (\ref{eq:generalsolutionN}), the necessary boundary conditions in Theorem \ref{theo:C1 cond at connection} are the following smooth pasting conditions.
If the both boundaries are finite, for each boundary, 
\begin{align}
  z^{\partial\{-\}}: &  
  \begin{cases}
    u_{n}^{\xi\{\beta\}}   (z^{\partial\{-\}})&=u_{n-m_1}^{\hat{\xi}^{(m_1)}\{-\}}   (z^{\partial\{-\}})- m_1 K \\
    u_{n}^{\xi\{\beta\}}{'}(z^{\partial\{-\}})&=u_{n-m_1}^{\hat{\xi}^{(m_1)}\{-\}}{'}(z^{\partial\{-\}})
  \end{cases},\qquad z^{\partial\{-\}}\in\stoppingRegionCandidate{\xi\hat{\xi}\{\alpha_i\}}{1}\cap\continuationRegionA{n-m_1}{\hat{\xi}^{(m_1)}\{-\}}, \label{eq:smooth_pasting_condition_j}\\
  z^{\partial\{+\}}: & 
  \begin{cases}
    u_{n}^{\xi\{\beta\}}   (z^{\partial\{+\}})&=u_{n-m_2}^{\bar{\xi}^{(m_2)}\{+\}}   (z^{\partial\{+\}})- m_2 K \\
    u_{n}^{\xi\{\beta\}}{'}(z^{\partial\{+\}})&=u_{n-m_2}^{\bar{\xi}^{(m_2)}\{+\}}{'}(z^{\partial\{+\}})
  \end{cases},\qquad z^{\partial\{+\}}\in\stoppingRegionCandidate{\xi\bar{\xi}\{\alpha_j\}}{1}\cap\continuationRegionA{n-m_2}{\bar{\xi}^{(m_2)}\{+\}}. \label{eq:smooth_pasting_condition_k}
  \intertext{Moreover, at the same time from (\ref{def:C_u}), the following should be satisfied.}
  & \forall\hat{\xi}\in\controlspace(\xi), u_{n}^{\xi\{\beta\}} > v(z, \hat{\xi}, n-1)  - K  \text { on }\continuationRegionA{n}{\xi\{\beta\}} \label{ineq:VI inequality cond}
\end{align}
The four undetermined coefficients of  the above system of four equations are 
\begin{align}
  \CC{\xi}{\pm\{\beta\}},
  z^{\partial\{-\}},
  z^{\partial\{+\}}\nonumber
\end{align}

According to Algorithm \ref{algo:system_freeBoundary}, 
if the left-most reference region is finite, solve another free boundary problem with the infinite continuation region on its immediate left region of the left-most reference region including $-\infty$ with the asymptotic condition.
In the same way, if the right-most reference region is finite, solve the free boundary problem with the infinite continuation region on its immediate right region of the right-most reference region including $\infty$ with the asymptotic condition.

On the other hand, if the boundary $z^{\partial\{-\}}$ or $z^{\partial\{+\}}$ is infinite, the corresponding boundary condition is the following asymptotic condition:
\begin{align}
  \valuefuncB{v}{z}{\xi}{n} \to \valuefuncB{v}{z}{\xi}{0}=\vfhat\quad (z\to\pm\infty),\qquad  \xi\in\positionspace, n=1, 2, \cdots \label{lim:infty} 
\end{align}
This condition corresponds to the case where the action of the continuation is the optimal when $z$ approaches $\pm\infty$. In this case the value function with $n$ switches approaches the value function with zero switches.
If $z^{\partial\{-\}} \text{ or } z^{\partial\{+\}}=\pm\infty$,
as $\CC{\xi}{\mp\{\beta\}}=0$, for each $\{\beta\}\in\Xi$,
\begin{align}
  u_n^{\xi\{\beta\}}(z)&=\CC{\xi}{\pm\{\beta\}}\HermiteNu{\pm z}+\vfhat, \quad z\in\continuationRegionA{n}{\xi\{\beta\}} \text{ (double sign in the same order) } \label{eq:general_u_infty}
\end{align}

From Theorem \ref{theo:piecewiseClassicalViscositySolution}, if the above solution is found, the associated function on $\real{}$ composed of the connected piecewise classical solutions results in the unique viscosity solution which coincides with the value function.

\begin{omitlevel1}
また、微分方程式の一般解及び境界条件を得ただけでは、最適解に関する状態変数空間上の継続領域・スイッチ領域の構造は明らかにならず、
条件を満たす領域を具体的に求めてみる必要がある。
本稿では、具体的な問題に対する具体的な解の例として２つの最適ペア運用戦略問題を取り上げる。１つは、
\citeN{suzuki2016optimal}で取り上げられているExample \ref{ex:identifyingSwitchingRegions1}で、他方は\citeN{suzuki2020optimal}で取り上げられているExample \ref{ex:identifyingSwitchingRegions2}である。後者は前者の線型評価関数を二次関数に拡張したモデルである。つまりリスク回避係数を$\lambda=0$とする前者は、それを$\lambda\geq 0$とする後者の特別な場合と捉えることができる。それ以外のパラメータは共通である。

\subsubsubsection{エルミート関数とその実装} ~\\

解析解という言葉は、通常、代数関数や初頭関数、特殊関数などの公知の関数を使って表現できる解という意味で用いられることが多く、単なる収束冪級数を表すいわゆる解析関数よりは初等の関数であるべきである。
本稿の取り扱う一般解で用いられている基底関数$\{\HermiteNu{z}, \HermiteNu{-z}\}$は、合流型超幾何関数(confluent hypergeometric function)${}_1F_1$ を用いて表現される次のような自由度$\nu$のエルミート関数(Hermite function)$\HermiteNu{z}$である(\citeN{Lebedev:72}, P. 285)。
\begin{align}
    \HermiteNu{z}&=\dfrac{2^\nu\Gamma\left(\frac{1}{2}\right)}{\Gamma\left(\frac{1-\nu}{2}\right)} {}_1F_1 \left(-\frac{\nu}{2}, \frac{1}{2}; z^2\right) + \dfrac{2^\nu\Gamma\left(-\frac{1}{2}\right)}{\Gamma\left(-\frac{\nu}{2}\right)}z {}_1F_1 \left(\frac{1-\nu}{2}, \frac{3}{2}; z^2\right)\label{eq:def_Hermite} 
\end{align}
ただし、合流型超幾何関数${}_1F_1$は、
\begin{align}
    {}_1F_1(\alpha, \gamma; z)&=\sum_{k=0}^\infty\dfrac{(\alpha)_k}{(\gamma)_k}\dfrac{z^k}{k!}, \abs{z}<\infty, \gamma\ne0, 1, 2, \cdots,  \label{def:confluent} \\
    (\lambda)_k&=\dfrac{\Gamma(\lambda+k)}{\Gamma(\lambda)}=\lambda(\lambda+1)\cdots(\lambda+k-1), k=1, 2, \cdots, \nonumber\\
    (\lambda)_0&=1\nonumber
\end{align}
で与えられる関数である。
超幾何関数は冪級数を用いて定義されており解析関数ではあるが、初等関数では表されず、公知と呼べるほど性質は明らかでない。
この関数の値を求めるためには冪級数計算のための数値計算を行う必要があり、その意味においては、数値解に近い解析解であるといえる。
\end{omitlevel1}

\section{Calculation of the Value Function by a Computer} \label{sec:VF_algorithm_datastr}
Following the algorithm presented in Section \ref{sec:Application of Algorithm}, we can calculate the value functions, which are composed only of piecewise classical analytical functions, even though the solutions are in the sense of viscosity. We calculate the value functions by a computer and plot the graphs in Figure \ref{chart:nonlinear_basic}, showing the relations between value functions $v(z, \xi , n), \xi=\{-1, 0, 1\}$ before ($n=1$) and after ($n=0$) the switching. 
The transition between the value functions $\valuefuncB{v}{z}{\xi}{n} \to \valuefuncB{v}{z}{\hat{\xi}}{n-1}$ is illustrated.
\begin{figure}[htbp]
\begin{center}
  \includegraphics[scale=0.7]{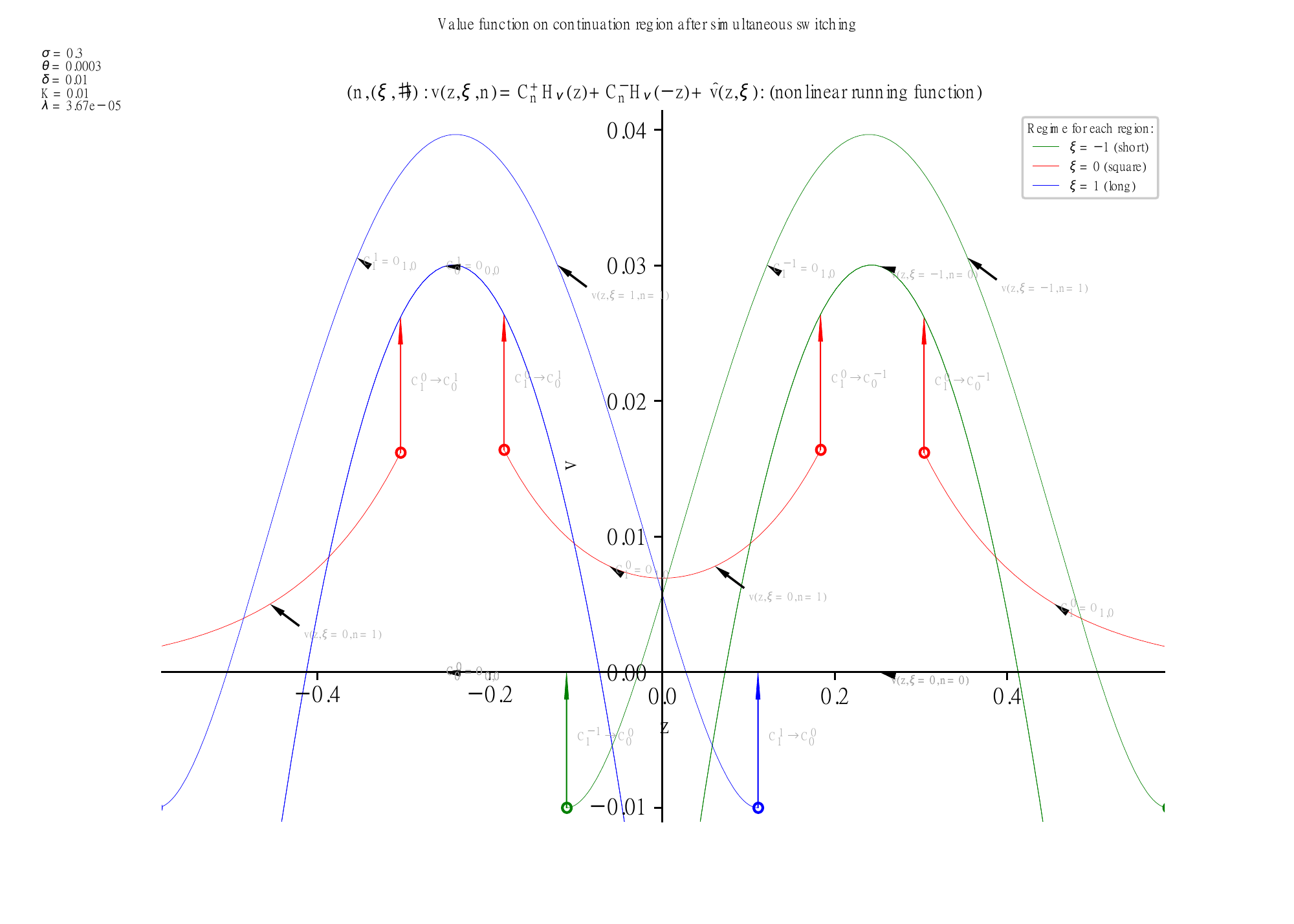}   
  \caption{Relations between value functions $v(z, \xi , n), \xi=\{-1, 0, 1\}$ before and after the switching ($n=1\to 0$)}
 \label{chart:nonlinear_basic}
\end{center}
\end{figure}

In this chart, each value function $v(z, \xi , 1), \xi=\{-1, 0, 1\}$ over the continuation region is shown transitioning to $v(z, \hat{\xi} , 0)=\hat{v}(z, \hat{\xi}), \hat{\xi}=\{-1, 0, 1\}$ through regime switching from $\xi$ to $\hat{\xi}\in\controlspace(\xi)$.
This geometric-analytic relationship corresponds to the smooth pasting conditions—i.e., $C^1$-class continuity conditions connecting the functions before and after switching—at the connection represented by equation~(\ref{eq:generalsolutionN}) on $\stoppingRegion{n}$ on the condition of equations~(\ref{eq:smooth_pasting_condition_j}) and~(\ref{eq:smooth_pasting_condition_k}).

The figure is color-coded according to the three regimes.
The value function for $\xi=0, n=0$ coincides with the horizontal axis and is thus not visible.
Vertical arrows indicate the points at which switching occurs, and the length of each upward vector corresponds to the transaction cost.

According to the definition of the value function, it includes the expected cost from the current point onward, but past realized costs are excluded at the moment the event passes.
Therefore, transaction costs are embedded in the value function as expectations just before switching, but once paid, they are removed as past events, resulting in an instantaneous upward jump in the value function equal to the transaction cost.

For example, the central red curve in Figure \ref{chart:nonlinear_basic} represents the value function in the continuation region $\continuationRegionA{1}{0}$ for $\xi=0, n=1$.
At the boundaries of this region, switching occurs to $\hat{\xi}=-1$ (green) on the right, or $\hat{\xi}=1$ (blue) on the left, transitioning to the continuation region $\continuationRegionA{0}{\hat{\xi}}=\real{}$ for $n=0$.
Note that $v(z, \hat{\xi} , 0)=\hat{v}(z, \hat{\xi}),\; \hat{\xi}\in\positionspace$ (defined on equation (\ref{eq:v0})) is defined over the entire real line $\real{}$ as its continuation region.

The conditions for switching follow the procedures given in Theorem~\ref{theo:piecewiseClassicalViscositySolution}:
\begin{enumerate}
  \item According to the boundary condition in equations~(\ref{eq:smooth_pasting_condition_j}) and~(\ref{eq:smooth_pasting_condition_k}), the curve $v(z, \xi , 1)$ shifted upward by $K$ coincides with and is tangent to the corresponding curve $v(z, \hat{\xi} , 0)$ at the switching point (i.e., satisfies a $\mathcal{C}^1$-class condition).
  \item According to equation~(\ref{ineq:VI inequality cond}), the upward-shifted curve $v(z, \xi , 1) + K$ does not fall below any other subordinate value function $v(z, \hat{\xi} , n-1)$ for all $\hat{\xi}\in\controlspace(\xi)$.
\end{enumerate}

These conditions are confirmed to hold for each of the three disjoint continuation regions $\continuationRegionA{n}{\xi\{\beta_1\}}, \continuationRegionA{n}{\xi\{\beta_2\}}, \continuationRegionA{n}{\xi\{\beta_3\}}$ comprising $\continuationRegionA{n}{\xi}$ with $\xi=0, n=1$.
Similar confirmations can be made for all other regimes $\xi=\pm 1, n=1$.

Thus, the procedure for constructing $v(z, \xi , 1), \xi\in\positionspace$ from the known $v(z, \xi , 0)=\hat{v}(z, \xi)$ is verified.
Since $v(z, \xi , 0)$ is given, the recursive procedure to obtain the value function for any $n \geq 0$ is also confirmed.

In this way, using the switching relation in equations~(\ref{eq:smooth_pasting_condition_j}) and~(\ref{eq:smooth_pasting_condition_k}), the complete value function graph—including values in switching regions—over the continuation region is presented in Figure~\ref{chart:nonlinear_basic_vf}.
\begin{figure}[htbp]
  \begin{center}
    \includegraphics[scale=0.7]{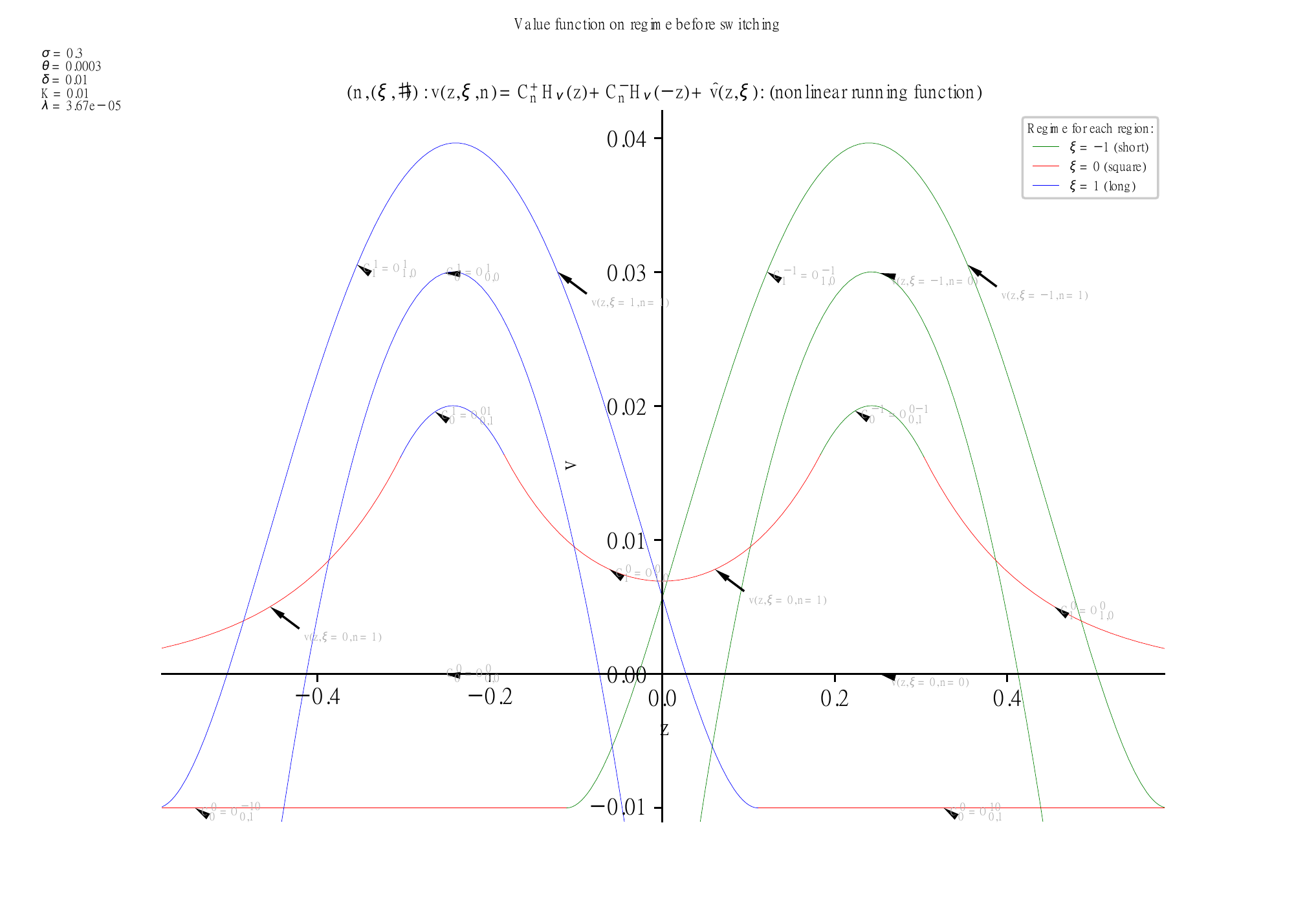}
    \caption{Value functions $v(z, \xi , n), \xi=\{-1, 0, 1\}, n=0, 1$}
    \label{chart:nonlinear_basic_vf}
  \end{center}
  \begin{flushleft}
    \small
  \end{flushleft}
\end{figure}

By repeating this procedure, the value function can be computed for any 
$n\geq 0$.
As an example, the results up to $n =5$ are illustrated in Figure \ref{chart:nonlinear_vf}.

\begin{figure}[htbp]
  \begin{center}
    \includegraphics[scale=0.7]{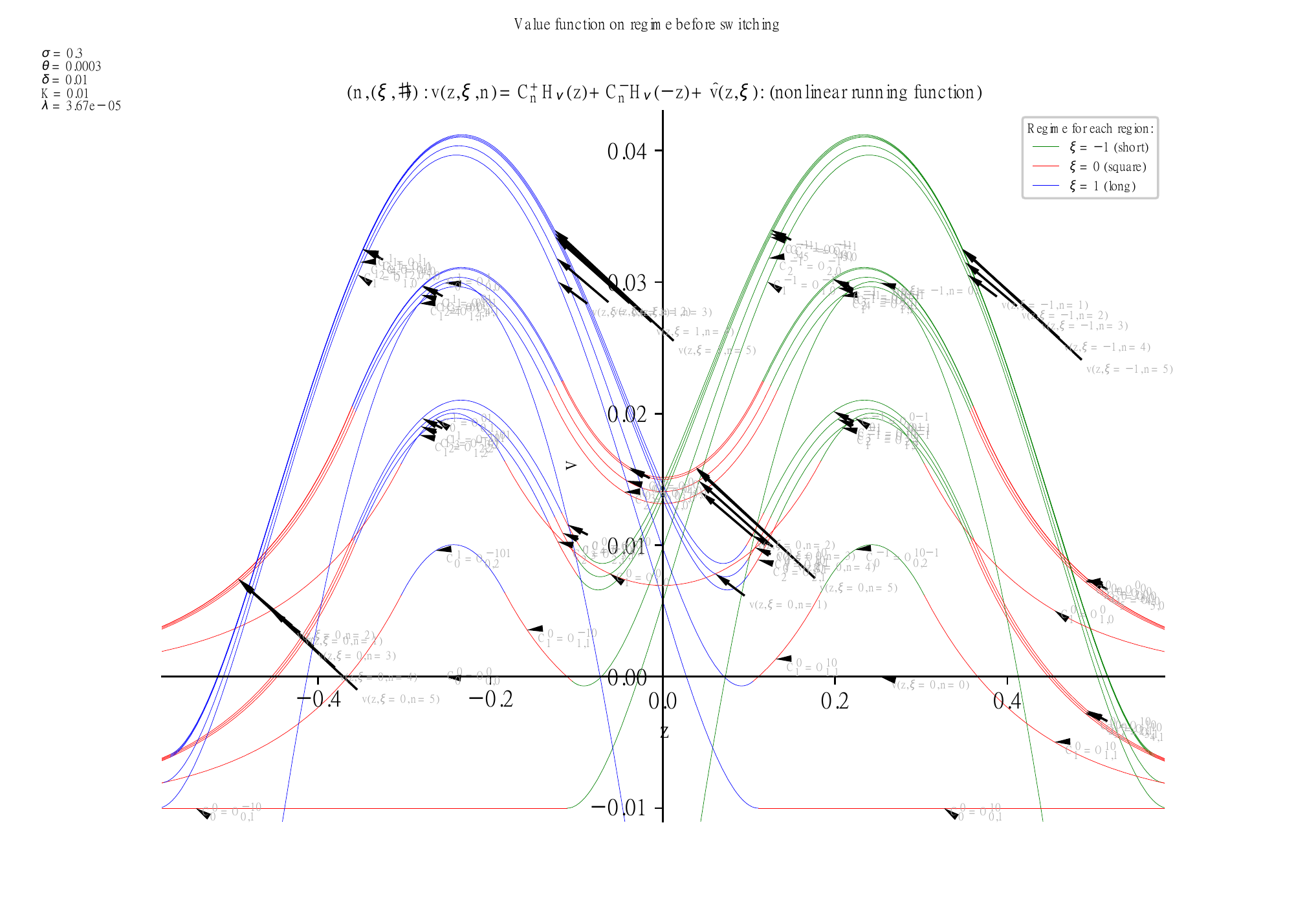}  
    \caption{Value function $v(z, \xi , n), \xi=\{-1, 0, 1\}, n=0, 1, \cdots , 5$}
    \label{chart:nonlinear_vf}
  \end{center}
\end{figure}

All the calculations to solve the viscosity solutions to the variational inequalities are done by Python programs with many special package libraries of Python.
We have not used any numerical solvers, but only plot analytical functions analytically solved by PDEs with special function packages of Python.
Below, Table~\ref{tab:python_package} lists the standard and external Python libraries used in the analysis.
\begin{table}[htbp]
  \begin{center}
    \begin{tabular}[c]{l l l}
      \hline
      Package Name & Purpose & Usage in Analysis \\
      \hline\hline
      \texttt{math}       & Mathematical functions     & Basic mathematical computations \\
      \texttt{numpy}      & Multi-dimensional arrays   & Array and matrix operations \\
      \texttt{pandas}     & Data analysis              & Data manipulation \\
      \texttt{matplotlib} & Visualization              & Chart and graph plotting \\
      \texttt{scipy}      & Special functions          & Calling special functions\\
      \texttt{sqlite3}    & Database                   & Managing a database of optimal solutions \\
      \texttt{logging}    & Debugging                  & Managing output of debug information \\
      \texttt{portion}    & Interval arithmetic        & Set operations over continuation regions \\
      \hline
    \end{tabular}
    \caption{Main Python standard and external packages used in the analysis}
    \label{tab:python_package}
  \end{center}
\end{table}

\begin{omitlevel1}
\section{値関数計算アルゴリズムとデータ構造} \label{sec:VF_algorithm_datastr}

本節で値関数の計算プログラムを計算機上に実装する際の計算アルゴリズムとデータ構造に関して述べる。
本稿のモデルは連続同時スイッチの枠組みを採用している。すなわちlong ($\xi=1$)からsquareを通り越してshort ($\xi=-1$)へのスイッチ(レジーム遷移)をモデル化する際、そのような直接のレジーム遷移を禁ずる代わりに$\xi=1$からsquare ($\xi=0$)へのスイッチと$\xi=0$から$\xi=-1$へのスイッチを同時に(経過時間0で)連続して行うことを許容する。

\subsection{スイッチの入れ子構造} \label{sec:スイッチの入れ子構造}
スイッチ権利数$n$を持つ最適スイッチ問題の構造を考慮すると、式(\ref{eq:generalsolutionN})にみられるように、最適スイッチ直後は再帰的にスイッチ権利数$n-1$を持つ同様の最適スイッチ問題に帰着する。すなわち値関数$\valuefuncB{v}{z}{\xi}{n}$を求める際、継続領域$\continuationRegionA{n}{\xi\{\beta_i\}}$からはみ出し、レジーム$\hat{\xi}\in\controlspace(\xi)$にスイッチされた先の領域では、スイッチ権利数$n$を$n-1$に置き換えた値関数$\valuefuncB{v}{z}{\hat{\xi}}{n-1}$を用いればよい。 $n=0$による初期値は式(\ref{eq:v0})で与えられる。
つまりスイッチに伴い、値関数は式(\ref{eq:generalsolutionN})の2行目のような、スイッチ権利数$n$を序数とする漸化式を形成している。
権利数$n$の値関数を計算する際、継続領域外では、既知の値関数$\valuefuncB{v}{z}{\hat{\xi}}{n-1}$を式(\ref{eq:generalsolutionN})の2行目式にあてはめれば計算ができる。

式(\ref{eq:R_decompose_to_C}), (\ref{eq:decomp_continuation_region})を組み合わせると、計算処理という観点からすると、値関数$\valuefuncB{v}{z}{\xi}{n}$の計算は、実質的に各継続領域$\continuationRegionA{n}{\xi\{\beta_i\}}$上に限って行えばよいことが分かる。
スイッチ権利数$n$の計算をする際には、スイッチ権利数$n-1$の値は既知である。計算処理アルゴリズムは、フィボナッチ数を計算するのと同じような再帰的な手順で行えばよい。
ただしレジーム$\xi\in\positionspace$が複数あり、スイッチのたびにそれが入れ替わるような、系列の交差があるような漸化式である。
特定の$\valuefuncB{v}{z}{\xi}{n}$だけ求めるのであれば、再帰的に求めるやり方もあるが、全$n$を網羅的に求めるのであれば、$n=0, 1, \cdots$という順序で計算するやり方の方が無駄がない。
この場合、$n$の計算をする際、$n-1$以前の計算結果は蓄積されている必要があるため、何らかのデータベースの仕組みも必要である。

\subsection{継続領域間の引継ぎ構造} 
これまでの考察により、スイッチ権利数$n$の値関数を求めるには、権利数$n$の各継続領域$\continuationRegionA{n}{\xi\{\beta_i\}}$中心に決定していけばよいことがわかった。
継続領域境界でのスイッチ発生以降は既知の関数$\valuefuncB{v}{z}{\hat{\xi}}{n-1}$に引き継げばよいからである。
スイッチ先が権利数$n-1$の継続領域$\continuationRegionA{n-1}{\hat{\xi}}$内部のとき関数$\valuefuncB{v}{z}{\hat{\xi}}{n-1}$の計算が未だであれば再帰的に同様の手順を繰り返せばよい。
スイッチ先が権利数$n-1$のスイッチ領域の場合は更に同時連続スイッチが発生する。
値関数$\valuefuncB{v}{z}{\xi}{n} \to \valuefuncB{v}{z}{\hat{\xi}}{n-1}$間の引継ぎの様子を図示したのが図\ref{chart:nonlinear_basic}である。
\begin{figure}[htbp]
\begin{center}
  \includegraphics[scale=0.7]{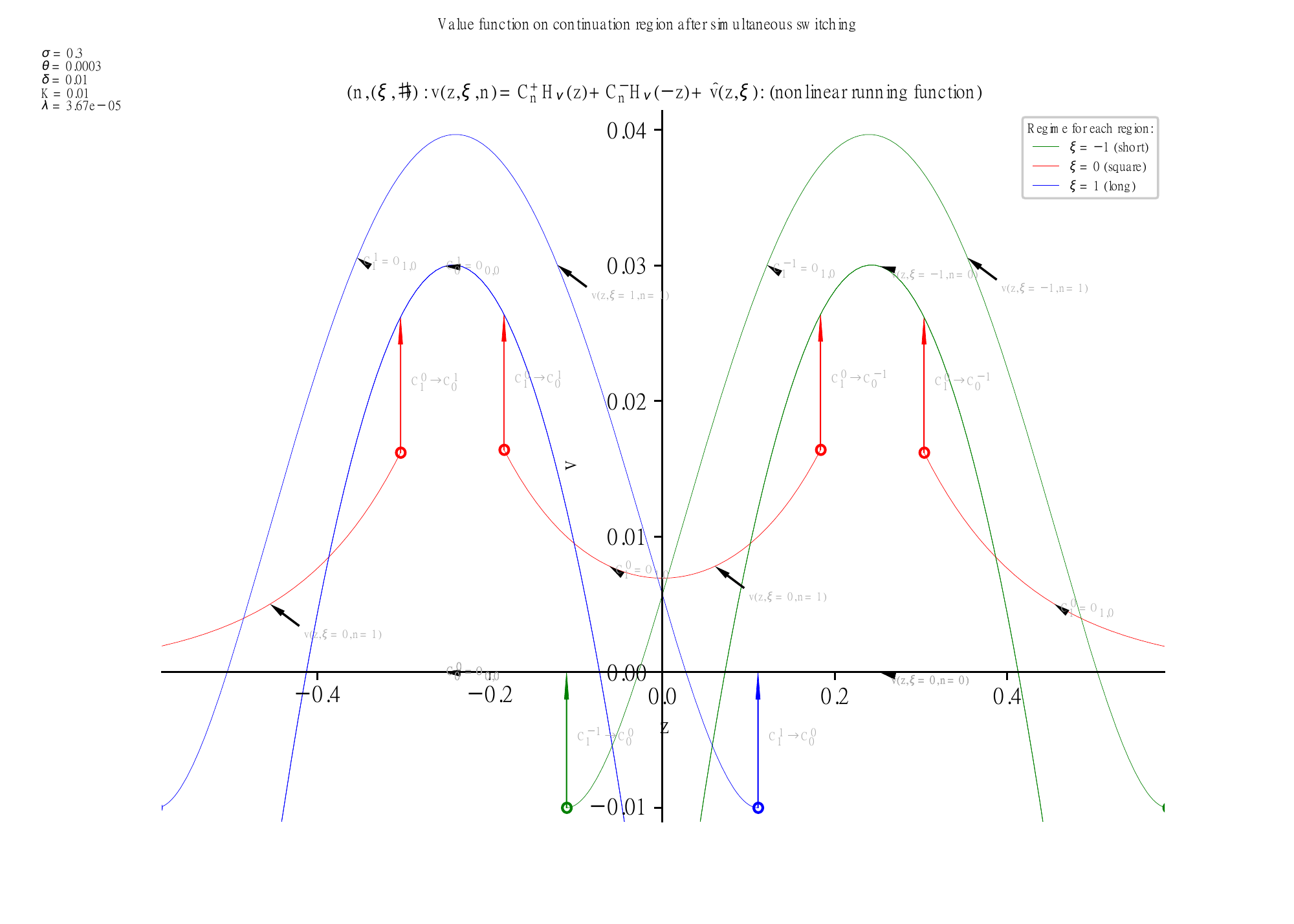}   
  \caption{値関数$v(z, \xi , n), \xi=\{-1, 0, 1\}$の$n=1\to 0$へのスイッチ前後の関係}
 \label{chart:nonlinear_basic}
\end{center}
\end{figure}
 
この例は本稿の非線形評価関数問題の実際の解を数値計算で求めた結果でもある。
このチャートでは、継続領域上のそれぞれの値関数$v(z, \xi , 1), \xi=\{-1, 0, 1\}$がレジーム$\hat{\xi}\in\controlspace(\xi)$へとスイッチすることによる値関数$v(z, \hat{\xi} , 0)=\hat{v}(z, \hat{\xi}), \hat{\xi}=\{-1, 0, 1\}$への引継ぎの様子を示している。
すなわち、この解析幾何学的観点に基づく関係は、解析式としては、式(\ref{eq:generalsolutionN})のスイッチ側の式に基づいて導かれた、スイッチ前後の関数の接続条件となる$C^1$-級条件(smooth pasting condition)を表す、式(\ref{eq:smooth_pasting_condition_j}), (\ref{eq:smooth_pasting_condition_k})に対応している。	

この図は、3種類のレジーム毎に色分けされている。
$\xi=0, n=0$の値関数は図の横軸に一致して重なっていて見えない。
垂直矢印は、その点でスイッチが発生することを示す。各上向きベクトルの長さは取引コスト分である。
値関数の定義によると、その時点から見て将来発生し得る期待コストは含まれているが、実績過去分はイベントが去った瞬間に剥落するように定義されており、取引コストはスイッチ直前まではその期待値として値関数に組み込まれているがコストを払った瞬間に過去イベントとなり、評価関数から除去されるため取引コスト分だけ値関数はスイッチの瞬時に上昇する。

例えば中央の赤の曲線は$\xi=0, n=1$の継続領域$\continuationRegionA{1}{0}$の値関数である。領域端点にて、右側であれば$\hat{\xi}=-1$(緑)、左端点あれば$\hat{\xi}=1$(青)の$n=0$の継続領域$\continuationRegionA{0}{\hat{\xi}}=\real{}$にスイッチされる。なお$v(z, \hat{\xi} , 0)=\hat{v}(z, \hat{\xi}),\; \hat{\xi}\in\positionspace$は特別に実数全域$\real{}$で継続領域となっている。
スイッチ発生のための条件はTheorem \ref{theo:systemequation_viscosity}で示した手順書通りであるが、
\begin{enumerate}
  \item 式(\ref{eq:smooth_pasting_condition})の境界条件により、$v(z, \xi , 1)$を$K$だけ上方シフトさせた曲線と、対応する$v(z, \hat{\xi} , 0)$の曲線はスイッチする点で一致しかつ接する($\continuousFunction{1}$-級条件)。
  \item 式(\ref{cond:geq_u})により、$v(z, \xi , 1)$を$K$だけ上方シフトさせた曲線は、他のいかなるスイッチ遷移可能なレジームによる下位の値関数$v(z, \hat{\xi} , 1-1), \forall \hat{\xi}\in\controlspace(\xi) $も下回ることがない。
\end{enumerate}
となっていることが確認できる。
$\xi=0, n=1$の継続領域$\continuationRegionA{n}{\xi}$は、3個の互いに素な連結継続領域、$\continuationRegionA{n}{\xi\{\beta_1\}}, \continuationRegionA{n}{\xi\{\beta_2\}}, \continuationRegionA{n}{\xi\{\beta_3\}}$に分かれているが、それぞれについて条件が遵守されている。
同様のことが他の全てのレジーム$\xi=\pm 1$でも確認できる。
こうして値関数$v(z, \xi , 0), \xi\in\positionspace$から$v(z, \xi , 1), \xi\in\positionspace$を構成する手順が確認できた。
$v(z, \xi , 0)=\hat{v}(z, \xi)$は既知であるため、再帰的手順に従い任意の$n\geq 0$の値関数を求める手順も確認できた。

こうして、式(\ref{eq:smooth_pasting_condition})のスイッチ前後の値関数の関係式により、継続領域上の各値関数グラフにスイッチ領域上の値を付け加えた完全な値関数のグラフを図\ref{chart:nonlinear_basic_vf}に示す。
\begin{figure}[htbp]
  \begin{center}
    \includegraphics[scale=0.7]{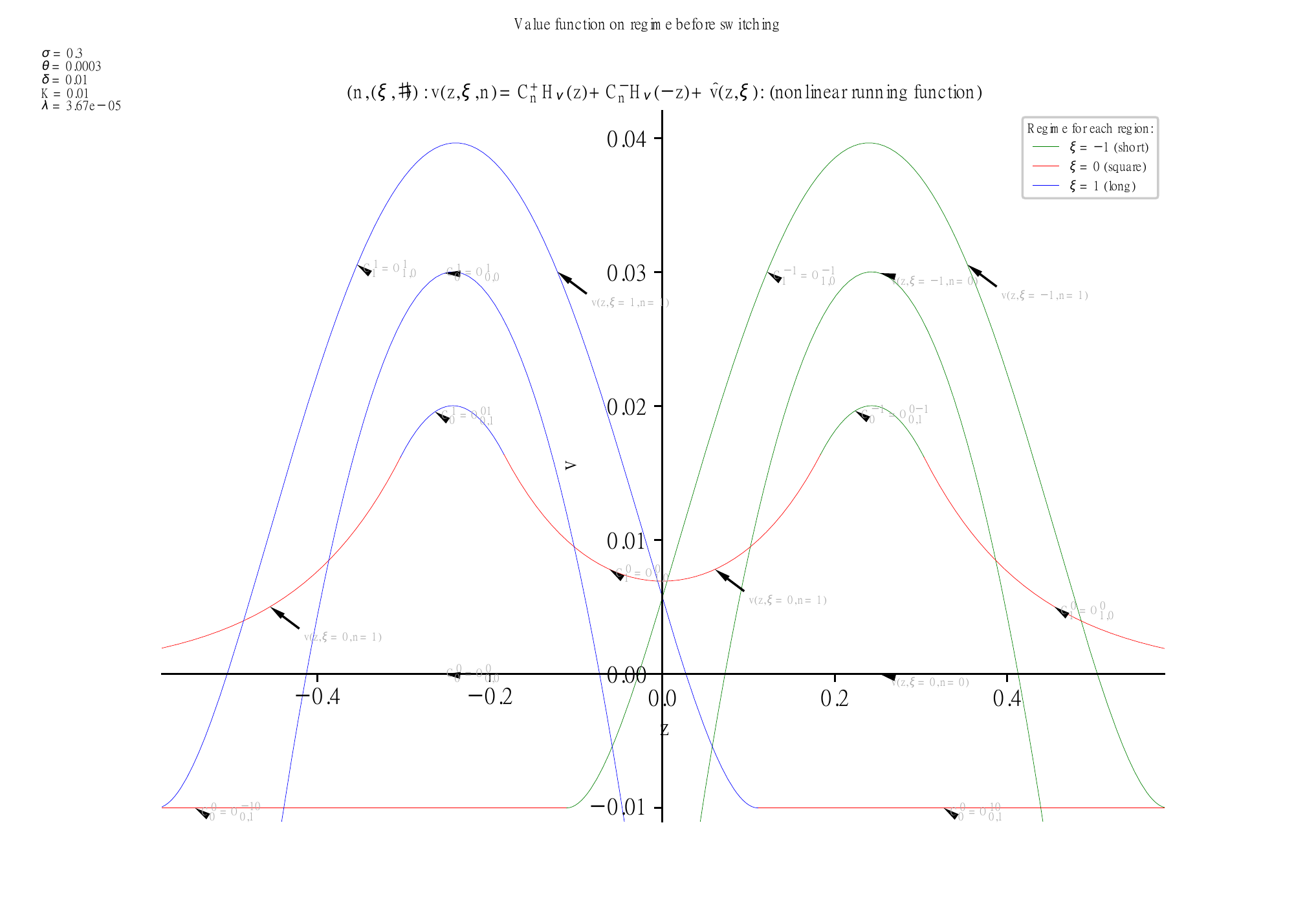}
    \caption{値関数$v(z, \xi , n), \xi=\{-1, 0, 1\}, n=0, 1$}
    \label{chart:nonlinear_basic_vf}
  \end{center}
  \begin{flushleft}
    \small
値関数$v(z, 0 , 0)\equiv 0$の赤線は水平軸に重なっている。
値関数$v(z, \pm 1 , 0)=\hat{v}(z, \pm 1)$は左右の放物線。
中段にある左右対称な曲線が値関数$v(z, 0 , 1)$。
放物線の上に覆い被さっているのが$v(z, \pm 1 , 1)$。
  \end{flushleft}
\end{figure}

各曲線とも、ベースとなる最終継続領域の属するレジームの色で色分けしている。
すなわち、式(\ref{eq:R_decompose_to_C})のように各$\xi^{(0)}\in\positionspace$につき、最終連続同時スイッチ先レジーム$\xi^{(m)}$による$\stoppingRegionA{n}{\xi^{(0)}\cdots{\xi}^{(m)}} \cap \continuationRegionA{n-m}{{\xi}^{(m)}}$で分割し、$\continuationRegionA{n-m}{{\xi}^{(m)}}$が最終継続領域となり、曲線の形状はこの継続領域上の古典解によって形成される。
\ref{sec:スイッチの入れ子構造}節の議論によりスイッチ領域内でも、実質的には最終スイッチ先の継続領域上での描画手順を用いてグラフを描画している。
これらの色の切れ目が領域相互の境界であり、つまりそれが最適戦略におけるスイッチ発生箇所に相当する。

この手順を繰り返すことにより任意の$n$について値関数の計算を行うことができる。試しに$n=5$までを描画したのが図\ref{chart:nonlinear_vf}である。
\begin{figure}[htbp]
  \begin{center}
    \includegraphics[scale=0.7]{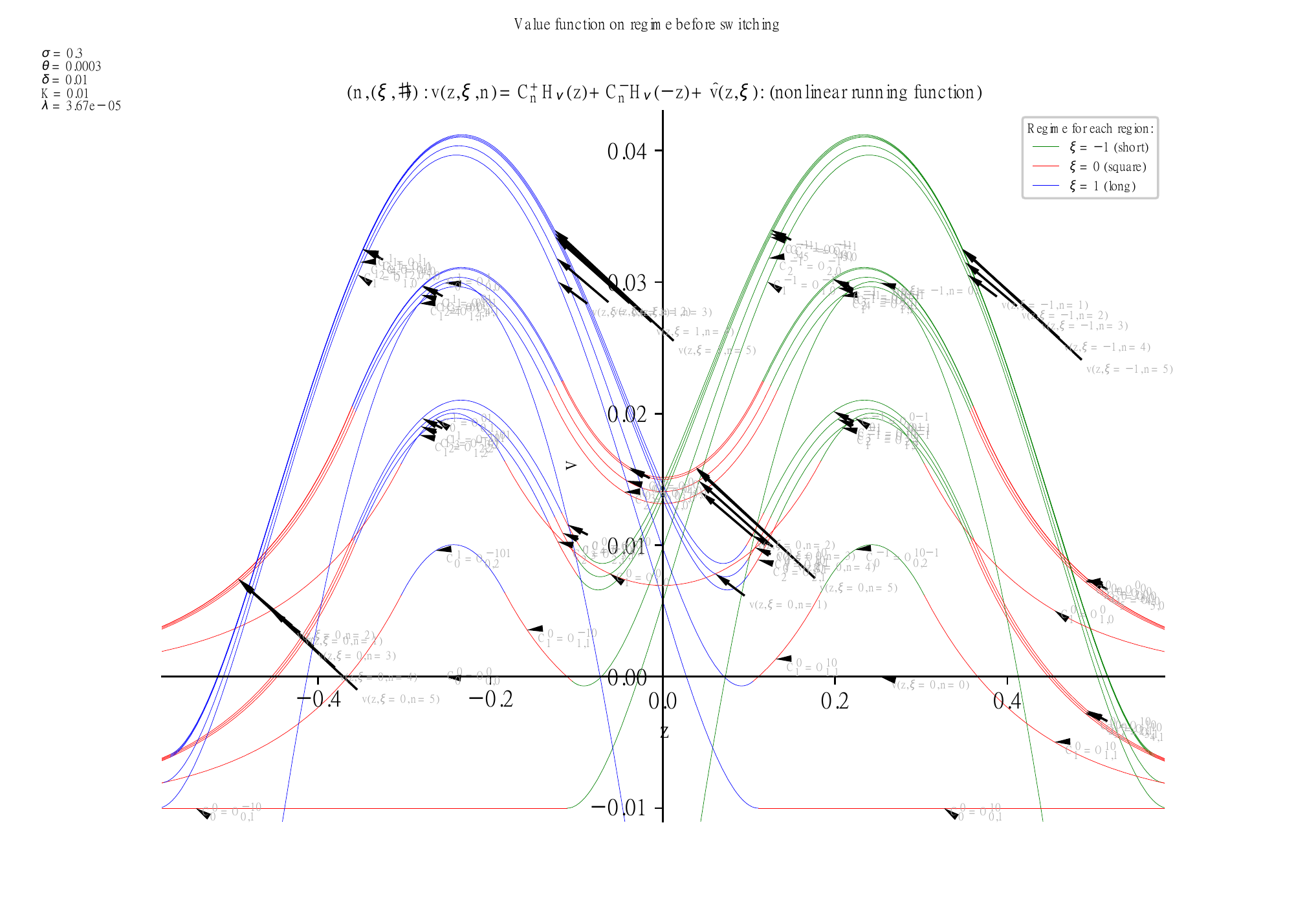}
    \caption{値関数$v(z, \xi , n), \xi=\{-1, 0, 1\}, n=0, 1, \cdots , 5$}
    \label{chart:nonlinear_vf}
  \end{center}
\end{figure}
もはや情報量が多すぎて図が煩雑であるが、チャート自体はEPSファイル形式で作成しており、ベクトル画像のため文書は任意倍率に拡大可能である。この文書が電子ファイルであれば細部に渡り内容が検証可能であろう。

\subsection{初期値問題と解の管理} 
処理全体の中で厄介なのは、式(\ref{eq:smooth_pasting_condition})による継続領域境界の両端による4元連立の非線型連立方程式の求根処理であろう。通常は方程式の解を求める数値計算ライブラリーを呼ぶことになるであろう。
本稿ではニュートン法の関数を用いた。外部関数を呼ぶのは、内容に信頼がおける反面、トラブルが起こっても内部をデバックする自由度が限定的である。
収束計算を行うことになるが、どの程度の誤差以内であれば真の解と見做すのかの制御はユーザーに任されており、試行錯誤を要する。
さらに\ref{sec:Hermite_install}節で取り扱ったように、本稿は巨大な値をとり得るHermite関数を含む計算をしているため、誤差と見做すべき数値の水準も場合によって変化するため注意を要する。

ニュートン法を用いる上で最も知恵が求められるのは、初期値の入力である。この手法の場合、原則的には解の近傍を初期値として選択できないと収束しない。収束計算であるため、求めるべき真の値の近傍を初期値に入れられれば短時間で高精度の解が発見されるため理想的なのであるが、
当然ながら、不明な解を求めるために関数を呼んでいることを考えると、解の近傍は普通は予め分かるものではない。
ただ、再帰的な計算を繰り返す場合、$n$の値関数を計算する際には$n-1$の解を初期値に用いればうまく行くことが多い。
しかし全く新しい問題を初めて解くような場合で解が存在するかどうかが不明な場合には特に試行錯誤を繰り返すことになるであろう。
未知数が実数の場合、式(\ref{exp:float_range})で表されるような実数全域の範囲から解の場所を推測するのはほぼ不可能に近いであろう。
特にエルミート関数のような爆発しやすい関数を取り扱うような場合、解も常識的とならない場合もあり、初期値も見当がつかない恐れがある。

このように苦労して求めた解は蓄積しておくと、後に類似のパラメータを持つ方程式が現れた時、その解を初期値に用いることにより楽に収束計算できる場合がある。少なくとも何もヒントがない場合と比べると根発見の手掛かりとなる。
そのためパラメータのパターンと解の対応表を蓄積するようなデータベースを構築することが望ましい。
そうすることにより、長い目で見れば効率よく数値計算できる。本稿ではPython組み込みのデータベース、sqlite3パッケージ(表\ref{tab:python_package})を用いてこのデータベース機能を実装した。

\subsection{解析的に自明な解の導入による解の存在範囲の推定}  \label{sec:obvious_solution}
前節で、ニュートン法で方程式の解の探索の際に知恵を要するのが初期値の決定であることを説明した。
実数空間上でニュートン法の初期値を決定することは、求めるべき解そのものを求めることと大差なく、事前知識が皆無の場合、巨大な実数空間の中から解の位置を言い当てることは不可能である。
求解のための一つの提案は、解析的に自明な解を導入することである。すなわち、与えられたパラメータのいくつかの値を一時的に修正することにより、解析的に明らかに求まる解を用いて、そこから少しずつ修正したパラメータを元の問題の値に徐々に近づけつつその都度、前ステップで求められた解を次のステップの初期値に用いることにより、最終的に目的とする元のパラメータ・セットまでたどり着くという手法である。この手法は、パラメータと解の関係が連続的であるという想定に基づいている。
すなわち、パラメータ・セットが類似している方程式の解どうしは互いに近い位置関係にあると想定している。

例えば本稿の問題の場合、$n=1$の$z=z^{\partial}$の境界条件式(\ref{eq:smooth_pasting_condition})に一般解(\ref{eq:general_u})を当てはめると、
\begin{align} 
     &\CCn{\xi}{+\{\beta_i\}}{1}\HermiteNu{z^{\partial}}+\CCn{\xi}{-\{\beta_i\}}{1}\HermiteNu{-z^{\partial}}+  \vfhatxiz{\xi}{z^{\partial}} =\vfhatxiz{\hat{\xi}}{z^{\partial}} - K , \label{eq:bc_1st}\\
     &\CCn{\xi}{+\{\beta_i\}}{1}\HermiteNui{z^{\partial}}-\CCn{\xi}{-\{\beta_i\}}{1}\HermiteNui{-z^{\partial}}+ \valuefuncA{\widehat{v}'}{z^{\partial}}{\xi}/(2\nu) =\valuefuncA{\widehat{v}'}{z^{\partial}}{\hat{\xi}}/(2\nu). \label{eq:bc_2nd}
\end{align}
ここで、レジーム$\xi=0, \hat{\xi}=\pm 1$の場合を考えると、$\vfhat$の定義より、$\valuefuncA{\widehat{v}}{z^{\partial}}{\xi}=\valuefuncA{\widehat{v}'}{z^{\partial}}{\xi}=0$.

ここで、式(\ref{eq:bc_1st}), (\ref{eq:bc_2nd})を$\CCn{\xi}{\pm\{\beta_i\}}{1}$を2独立変数とする二元連立一次方程式と考え、その非斉次項が$0$となる次の場合を考えると自明な解のケースを作りだせる。すなわち、
\begin{align}
  \begin{cases}
    \vfhatxiz{\hat{\xi}}{z^{\partial}} &= K, \\
    \valuefuncA{\widehat{v}'}{z^{\partial}}{\hat{\xi}}&=0,
  \end{cases} \label{eq:obvious_case}
\end{align}
となるパラメータが存在すれば、線形独立な連立方程式の根の関係から、
$\CCn{\xi}{\pm\{\beta_i\}}{1}=0$という解析的に自明な解が得られ、数値計算の支援となる。以下のProposition に式(\ref{eq:obvious_case})を満たす条件を示す。

\begin{proposition}[自明な解($n=1, \xi=0, \hat{\xi}=\pm1$)] \label{prop:obvious_solution}~\\
  $n=1, \xi=0, \hat{\xi}=\pm1$, および、
  \begin{align}
    \lambda=\lambda^* \equiv \dfrac{2\theta+\delta}{2\sigma^2(\theta+\delta)}\{ \sqrt{\delta\{K^2\delta(\theta+\delta)^2+\theta^2\sigma^2\}} - K\delta(\theta+\delta)\} > 0 \nonumber
  \end{align}
  のとき、解析的に自明な次の解が得られる。
  \begin{align}
    \begin{cases}
       \text{微分方程式の未定係数}:&\CCn{\xi}{\pm\{\beta_i\}}{1}=0, \\
       \text{自由境界の位置}:& z^{\partial\hat{\xi}}= -\hat{\xi} \dfrac{\sqrt{\theta}\theta(2\theta+\delta)}{2\lambda^*\sigma(\theta+\delta)}\;(\hat{\xi}=\pm 1)  
     \end{cases}\nonumber 
  \end{align}
\end{proposition}

\begin{proof}~\\
関数$\vfhatxiz{\hat{\xi}}{z}, \hat{\xi}=\pm 1$の解析幾何学的形状は上に凸の放物線となり、式(\ref{eq:obvious_case})のとき、放物線の頂点は$(z^{\partial}, K)$、軸は$z=z^{\partial}$である。

放物線$\vfhatxiz{\hat{\xi}}{z}, \hat{\xi}=\pm 1$の軸は、
\begin{align}
  z^{\partial\hat{\xi}}=-\hat{\xi} k_1/(2k_2)= -\hat{\xi} \dfrac{\sqrt{\theta}\theta(2\theta+\delta)}{2\lambda\sigma(\theta+\delta)} \label{eq:axis}
\end{align}
であり、頂点の$y$座標
\begin{align}
  b\equiv k_1^2/(4k_2)-k_0= \dfrac{{\theta^2}(2\theta+\delta)}{4\lambda(\theta+\delta)^2} - \dfrac{\lambda\sigma^2}{\delta(2\theta+\delta)} \nonumber
\end{align}
を$b=K$となるように、$\lambda=\lambda^*$を決めると、
\begin{align}
  \lambda^*= \dfrac{2\theta+\delta}{2\sigma^2(\theta+\delta)}\{ \sqrt{\delta\{K^2\delta(\theta+\delta)^2+\theta^2\sigma^2\}} - K\delta(\theta+\delta)\} > 0. \nonumber
\end{align}
このときの軸(式(\ref{eq:axis}))の位置は、
\begin{align}
  z^{\partial\hat{\xi}}=-\hat{\xi} \dfrac{\sqrt{\theta}\theta(2\theta+\delta)}{2\lambda^*\sigma(\theta+\delta)}\label{eq:axis*}
\end{align}
である。

つまり$\lambda=\lambda^*$とするとき、解析的に自明な解は、微分方程式の未定係数を$\CCn{\xi}{\pm\{\beta_i\}}{1}=0$として、また自由境界の位置$z^{\partial\hat{\xi}}\;(\hat{\xi}=\pm 1)$は式(\ref{eq:axis*})として得られる。

\end{proof}

式(\ref{eq:v0})より式(\ref{def:determinant})の判別式$D$は、
\begin{align}
  D=k_1^2-4k_2(K+k_0)=4 k_2 (b-K).  \label{def:determinant2}
\end{align}
これより、
\begin{align}
  \lambda\leq\lambda^* \Longleftrightarrow b\geq K \Longleftrightarrow D\geq 0 \Longrightarrow \stoppingRegionA{n}{0}\ne\phi, \forall n\geq 1.  \label{rel:lambda_b_D}
\end{align}
これより、$\lambda\leq\lambda^*$のとき任意の$n\geq 1$でレジーム$\xi=0$上にスイッチ領域の存在が保証される。
さもなくば、$n$によっては$\real{}$全域で継続領域ということもあり得る。
すなわち$\lambda=\lambda^*$は、全ての$n$に対して$\xi=0$から他のレジームにスイッチできる可能性のあるリスク回避係数の上限である。
リスク回避係数$\lambda$が$\lambda^*$を超すと、$\xi=0$から他のレジームへスイッチできる保証は消滅する。
また、式(\ref{rel:lambda_b_D})は式(\ref{eq:obvious_case})のとき等式で成立する。すなわち、
\begin{align}
  [Equation (\ref{eq:obvious_case})] \Longleftrightarrow   \lambda=\lambda^* \Longleftrightarrow b=K \Longleftrightarrow D=0.  \label{eq:lambda_b_D} 
\end{align} 
   
式(\ref{eq:lambda_b_D})の状態にある場合の値関数を図示したものが図\ref{chart:special_case_switching}である。
\begin{figure}[htbp]
  \begin{center}
    \includegraphics[scale=0.7]{Figure_special_switching.eps}
    \caption{自明な解によるスイッチの様子：$v(z, 0 , 1) \to \hat{v}(z, \hat{\xi}), \hat{\xi}=\pm 1$}
    \label{chart:special_case_switching}
  \end{center}
  \begin{flushleft}
    \small
一般のケースの図\ref{chart:nonlinear_basic_vf}を参照すると、$\continuationRegionA{1}{0}$は通常、3個の連結部分継続領域に分かれている。それらを左から順に$\continuationRegionA{1}{0\{\beta\}}, \beta=\beta_1, \beta_2, \beta_3$とおく。図\ref{chart:special_case_switching}では3個の領域は隙間なく配置されており、それらの境界は左右の放物線の軸上にある。隙間はないが穴が開いており、過程$Z$がそこに到達すると他のレジームへのスイッチ遷移が発生する。
図では継続領域$\continuationRegionA{0}{0}$を構成する、各連結継続領域$\continuationRegionA{1}{0\{\beta\}},\; \beta=\beta_1, \beta_2, \beta_3$の赤線は水平軸に重なっている。
  \end{flushleft}
\end{figure}
図\ref{chart:special_case_switching}の3個の連結部分継続領域$\continuationRegionA{1}{0\{\beta\}}, \beta=\beta_1, \beta_2, \beta_3$のいずれの有限端点においてもsmooth pastingの境界条件式(\ref{formula:v_in_C1}), (\ref{eq:smooth_pasting_condition}), (\ref{eq:obvious_case})が満たされていることが確認される。

図\ref{chart:special_case_switching}は図\ref{chart:nonlinear_basic}の特別な場合であり、より一般的な状態を表す図\ref{chart:nonlinear_basic}では$\continuationRegionA{1}{0}$を構成する個々の連結継続領域間に適度な隙間が空いているが、一方、特別な場合である、図\ref{chart:special_case_switching}では隙間の間隔は0である。すなわち実数$\real{}$上ほぼ全域が継続領域となる。ただし隙間間隔の面積が0であっても状態がそこへ到達すればスイッチが発生するという意味では隙間の両岸の領域は分断されたままである。

こうして$\lambda=\lambda^*$のときの自明の解は求められた。
数値計算上はこれを手掛かりとして、次に$\lambda^*$の近傍で、実際に与えられた$\lambda$の方向に$\lambda$の値を近づけながら対応する解を求める。その解を次のステップの問題の初期値として用いる。それを繰り返し、最終的に与えられた$\lambda$の解までたどり着くという手順である。
中間時点で得られたパラメータ・セットと対応する最適解の対応表というデータの管理をデータベースで行うことにより効率よく解を発見できる。こうして求解可能なパラメータ・セットのパターンを増やす度にデータベースにそれらの計算結果を蓄積することにより未知のパラメータのパターンに対しても解を早く探し出せるような環境を構築する。
\end{omitlevel1}

\section{Conclusion}\label{sec:conclusion}
In this paper, we have presented a novel approach to solving a variational inequality arising in an infinite-horizon optimal switching problem with simultaneous multiple switchings. Building upon the uniqueness result of the viscosity solution shown in \citeN{suzuki2020optimal}, we have demonstrated that, under appropriate assumptions, the viscosity solution can be explicitly constructed as a series of piecewise classical solutions. The key contribution lies in the establishment of necessary and sufficient smooth pasting conditions, which ensure that the concatenated classical solutions indeed satisfy the viscosity solution framework over the entire domain.

We proposed a practical algorithm for identifying all free boundary subproblems associated with the continuation and switching regions, and showed that this enables us to compute the complete value function recursively for any given number of remaining switches. The algorithm was applied to a pair-trading model based on an Ornstein–Uhlenbeck process. By solving each component analytically using special functions, we obtained a complete and explicit representation of the value functions without relying on numerical solvers.

The results not only provide a theoretical foundation for understanding the structure of the viscosity solution in such optimal switching problems but also offer a computationally feasible method for implementation in practice. The algorithmic framework and explicit formulas derived in this study can potentially be extended to more complex regime-switching models or multi-dimensional cases.

Future research may consider relaxing some of the assumptions, incorporating stochastic volatility, or applying the method to problems in energy economics, real options, and inventory control, where regime-switching behavior and transaction costs are inherent features.
 
  \section*{Acknowledgments}
  This research did not receive any specific grant from funding agencies in the public, commercial, or not-for-profit sectors.
  Also, the author is grateful to anonymous referees for their suggestions which have greatly improved the presentation of the paper.

\bibliographystyle{model5-names}  
\bibliography{optimize_general,switching,mean_reversion}

\begin{thebibliography}{12}
\expandafter\ifx\csname natexlab\endcsname\relax\def\natexlab#1{#1}\fi
\providecommand{\url}[1]{\texttt{#1}}
\providecommand{\href}[2]{#2}
\providecommand{\path}[1]{#1}
\providecommand{\DOIprefix}{doi:}
\providecommand{\ArXivprefix}{arXiv:}
\providecommand{\URLprefix}{URL: }
\providecommand{\Pubmedprefix}{pmid:}
\providecommand{\doi}[1]{\href{http://dx.doi.org/#1}{\path{#1}}}
\providecommand{\Pubmed}[1]{\href{pmid:#1}{\path{#1}}}
\providecommand{\bibinfo}[2]{#2}
\ifx\xfnm\relax \def\xfnm[#1]{\unskip,\space#1}\fi
\bibitem[{Crandall et~al.(1992)Crandall, Ishii \&
  Lions}]{CrandallIshiiLions:92}
\bibinfo{author}{Crandall, M.}, \bibinfo{author}{Ishii, H.}, \&
  \bibinfo{author}{Lions, P.-L.} (\bibinfo{year}{1992}).
\newblock \bibinfo{title}{User's guide to viscosity solutions of second order
  partial differential equations}.
\newblock {\it \bibinfo{journal}{Appeared in bulletin of the American
  mathematical society}\/},  {\it \bibinfo{volume}{27}\/},
  \bibinfo{pages}{1--67}.
\bibitem[{Fleming \& Soner(1993)}]{Fleming:93}
\bibinfo{author}{Fleming, W.~H.}, \& \bibinfo{author}{Soner, H.~M.}
  (\bibinfo{year}{1993}).
\newblock {\it \bibinfo{title}{Controlled Markov Processes and Viscosity
  Solutions}\/}.
\newblock \bibinfo{publisher}{Springer-Verlag, New York}.
\bibitem[{Guo \& Zhang(2005)}]{GuoZhang:05}
\bibinfo{author}{Guo, X.}, \& \bibinfo{author}{Zhang, Q.}
  (\bibinfo{year}{2005}).
\newblock \bibinfo{title}{Optimal selling rules in a regime switching model}.
\newblock {\it \bibinfo{journal}{IEEE transactions on automatic control}\/},
  {\it \bibinfo{volume}{50}\/}, \bibinfo{pages}{1450--1455}.
\bibitem[{Pemy \& Zhang(2006)}]{PemyZhang:06}
\bibinfo{author}{Pemy, M.}, \& \bibinfo{author}{Zhang, Q.}
  (\bibinfo{year}{2006}).
\newblock \bibinfo{title}{Optimal stock liquidation in a regime switching model
  with finite time horizon}.
\newblock {\it \bibinfo{journal}{Journal of Mathematical Analysis and
  Applications}\/},  {\it \bibinfo{volume}{321}\/}, \bibinfo{pages}{537--552}.
\bibitem[{Pham(2009)}]{Pham:09}
\bibinfo{author}{Pham, H.} (\bibinfo{year}{2009}).
\newblock {\it \bibinfo{title}{Continuous-time Stochastic Control and
  Optimization with Financial Applications (Stochastic Modelling and Applied
  Probability)}\/}.
\newblock \bibinfo{publisher}{Springer-Verlag}.
\bibitem[{Pham \& Vath(2007)}]{PhamVath:07}
\bibinfo{author}{Pham, H.}, \& \bibinfo{author}{Vath, V.~L.}
  (\bibinfo{year}{2007}).
\newblock \bibinfo{title}{Explicit solution to an optimal switching problem in
  the two-regimes case}.
\newblock {\it \bibinfo{journal}{SIAM J. Cont. Optim.}\/},  {\it
  \bibinfo{volume}{46}\/}, \bibinfo{pages}{395--426}.
\bibitem[{Pham et~al.(2009)Pham, Vath \& Zhou}]{PhamVathZhou:09}
\bibinfo{author}{Pham, H.}, \bibinfo{author}{Vath, V.~L.}, \&
  \bibinfo{author}{Zhou, X.~Y.} (\bibinfo{year}{2009}).
\newblock \bibinfo{title}{Optimal switching over multiple regimes}.
\newblock {\it \bibinfo{journal}{SIAM J. Cont. Optim.}\/},  {\it
  \bibinfo{volume}{48}\/}, \bibinfo{pages}{2217--2253}.
\bibitem[{Suzuki(2016)}]{suzuki2016optimal}
\bibinfo{author}{Suzuki, K.} (\bibinfo{year}{2016}).
\newblock \bibinfo{title}{Optimal switching strategy of a mean-reverting asset
  over multiple regimes}.
\newblock {\it \bibinfo{journal}{Automatica}\/},  {\it \bibinfo{volume}{67}\/},
  \bibinfo{pages}{33--45}.
\bibitem[{Suzuki(2018)}]{suzuki2018optimal}
\bibinfo{author}{Suzuki, K.} (\bibinfo{year}{2018}).
\newblock \bibinfo{title}{Optimal pair-trading strategy over long/short/square
  positions--empirical study}.
\newblock {\it \bibinfo{journal}{Quantitative Finance}\/},  {\it
  \bibinfo{volume}{18}\/}, \bibinfo{pages}{97--119}.
\bibitem[{Suzuki(2021)}]{suzuki2020optimal}
\bibinfo{author}{Suzuki, K.} (\bibinfo{year}{2021}).
\newblock \bibinfo{title}{Infinite horizon optimal switching regions for
  pair-trading strategy with quadratic risk aversion considering simultaneous
  multiple switching: a viscosity solution approach}.
\newblock {\it \bibinfo{journal}{Mathematics of Operations Research}\/},  {\it
  \bibinfo{volume}{46}\/}, \bibinfo{pages}{336--360}.
\bibitem[{Zervos(2003)}]{Zervos:03}
\bibinfo{author}{Zervos, M.} (\bibinfo{year}{2003}).
\newblock \bibinfo{title}{A problem of sequential entry and exit decisions
  combined with discretionary stopping}.
\newblock {\it \bibinfo{journal}{SIAM Journal on Control and Optimization}\/},
  {\it \bibinfo{volume}{42}\/}, \bibinfo{pages}{397--421}.
\bibitem[{Zhang \& Zhang(2008)}]{ZhangZhang:08}
\bibinfo{author}{Zhang, H.}, \& \bibinfo{author}{Zhang, Q.}
  (\bibinfo{year}{2008}).
\newblock \bibinfo{title}{Trading a mean-reverting asset: Buy low and sell
  high}.
\newblock {\it \bibinfo{journal}{Automatica}\/},  {\it \bibinfo{volume}{44}\/},
  \bibinfo{pages}{1511--1518}.

\end{thebibliography}

\end{document}